\documentclass{amsart}
\usepackage{amscd,graphicx,epsfig}
\usepackage[colorlinks=true, pdfstartview=FitV, linkcolor=blue, citecolor=blue, urlcolor=blue]{hyperref}
\usepackage{amssymb}
\usepackage[small,nohug,heads=vee]{diagrams} 
\usepackage{datetime}

\newtheorem{thm}{Theorem}[section]
\newtheorem*{thm*}{Theorem}

\newtheorem*{conj*}{Conjecture}

\newtheorem{prop}[thm]{Proposition}
\newtheorem{lem}[thm]{Lemma}
\newtheorem{cor}[thm]{Corollary}
\newtheorem*{cor*}{Corollary}

\theoremstyle{definition}
\newtheorem{defn}[thm]{Definition}
\newtheorem{exa}[thm]{Example}

\theoremstyle{remark}
\newtheorem*{rem*}{Remark}
\newtheorem{rem}[thm]{Remark}

\newcommand{\pp}{{\mathfrak p}}
\newcommand{\mm}{{\mathfrak m}}

\newcommand{\kk}{\Bbbk}

\newcommand{\ff}{{\mathfrak f}}

\newcommand{\name}[1]{\textsc{#1\/}}

\newcommand{\ZZ}{{\mathbb Z}}

\newcommand{\CC}{{\k}}

\newcommand{\Gm}{{\mathbb G}_{m}}
\newcommand{\kst}{{{\k}^*}}

\renewcommand{\AA}{{\mathbb A}}

\newcommand{\Atwo}{{\mathbb A}^{2}}
\newcommand{\Aone}{{\mathbb A}^{1}}

\newcommand{\NNN}{\mathcal N}
\newcommand{\SSS}{\mathcal S}
\newcommand{\PPP}{\mathcal P}
\newcommand{\VVV}{\mathcal V}
\newcommand{\simto}{\xrightarrow{\sim}}
\newcommand{\be}{\begin{enumerate}}
\newcommand{\ee}{\end{enumerate}}

\DeclareMathOperator{\Char}{char}

\DeclareMathOperator{\id}{id}

\DeclareMathOperator{\Mor}{Mor}

\DeclareMathOperator{\im}{im}
\DeclareMathOperator{\Spec}{Spec}
\DeclareMathOperator{\pr}{pr}
\DeclareMathOperator{\codim}{codim}
\DeclareMathOperator{\tdeg}{tdeg}
\DeclareMathOperator{\height}{ht}
\DeclareMathOperator{\VF}{Vec}
\DeclareMathOperator{\Der}{Der}
\DeclareMathOperator{\SL}{SL}
\newcommand{\SLtwo}{\SL_{2}}

\newcommand{\bbmat}{\begin{bmatrix}}
\newcommand{\ebmat}{\end{bmatrix}}
\newcommand{\bsmat}{\begin{smallmatrix}}
\newcommand{\esmat}{\end{smallmatrix}}
\newcommand{\Xbd}{{X_{\text{\it bd}}}}
\newcommand{\Vbd}{{V_{\text{\it bd}}}}

\newcommand{\into}{\hookrightarrow}

\renewcommand{\k}{\Bbbk}

\newcommand{\oA}{{\OOO(\AA^{7})}}

\newcommand{\oAga}{\oA^{\Ga}}

\newcommand{\xga}{{X/\!\!/\Ga}}

\newcommand{\PA}{\mathcal{P}_{\AA^{7}}}

\DeclareMathOperator{\LT}{LT}
\DeclareMathOperator{\Kdim}{Kdim}

\newcommand{\fxga}{{\ff}_{\xga}}
\newcommand{\bbb}[1]{\beta_{#1}}
\newcommand{\bd}{\text{\it bd}}
\newcommand{\cl}{\text{\it cl}}
\newcommand{\alg}{\text{\it alg}}

\newcommand{\graphX}{\Gamma_{X}}
\newcommand{\graphV}{\Gamma_{V}}
\newcommand{\graphXbar}{\overline{\graphX}}
\newcommand{\graphVbar}{\overline{\graphV}}

\newcommand{\OOO}{\mathcal O}
\newcommand{\quot}{/\!\!/}

\newcommand{\Ga}{{{\mathbb G}_{a}}}

\renewcommand{\phi}{\varphi}
\def \itt #1,#2:{\medskip\item[$\bullet$] %
     page\ \ignorespaces#1, line\ \ignorespaces#2:\ \ignorespaces}
\def\bullitem{\medskip\item[$\bullet$]}

\newcommand{\margin}[1]{}
\newcommand{\lab}[1]{\label{#1}}
\newcommand{\reff}[1]{\ref{#1}}

\newcommand{\Kt}{\k[\![t]\!]}
\newcommand{\Kf}{\k(\!(t)\!)}

\begin{document}
\title[Invariants and Separating Morphisms]{Invariants and Separating Morphisms for Algebraic Group Actions}

\author{Emilie Dufresne}
\address{Department of Mathematical Sciences, Durham University, Science Laboratories, South Rd., Durham DH1 3LE, UK}
\email{e.s.dufresne@durham.ac.uk}

\author{Hanspeter Kraft}
\address{Mathematisches Institut, Universit\"at Basel, Rheinsprung 21, CH-4051 Basel, Switzerland}
\email{Hanspeter.Kraft@unibas.ch}

\date{\today}

\thanks{The authors were partially supported by the SNF (Schweizerischer Nationalfonds)}

\begin{abstract}
The first part of this paper is a refinement of  \name{Winkelmann}'s work on invariant rings and quotients of algebraic groups actions on affine varieties, where we take a more geometric point of view. We show that the (algebraic) quotient $X\quot G$ given by the possibly not finitely generated ring of invariants  is ``almost'' an algebraic variety, and that the quotient morphism $\pi\colon X \to X\quot G$ has a number of nice properties. One of the main difficulties comes from the fact that the quotient morphism is not necessarily surjective.

These general results are then refined for actions of the additive group $\Ga$, where we can say much more. We get a rather explicit description of the so-called plinth variety and of the separating variety, which measures how much orbits are separated by invariants. The most complete results are obtained for  representations. We also give a complete and detailed analysis of \name{Roberts}' famous example of a an action of $\Ga$ on 7-dimensional affine space with a non-finitely generated ring of invariants.
\end{abstract}

\maketitle

\section{Introduction}
In all classification problems {\it invariants\/} play an important r\^ole. They let one distinguish non-equivalent objects, characterize specific elements, or detect certain properties. For instance, the genus of a curve determines a smooth curve up to birational equivalence, and the discriminant of a polynomial tells us whether it has multiple roots. In the algebraic setting, we often can reduce the classification problem to the following general situation. There is an algebraic variety $X$ representing the objects, and an algebraic group $G$ acting on $X$ such that two objects $x,y \in X$ are equivalent if and only if they belong to the same orbit under $G$. In this case the classification problem amounts to describing the orbit space $X/G$. Clearly, $X/G$ inherits some properties from $X$: it has a topology and the (continuous) functions on $X/G$ correspond to the (continuous) $G$-invariant functions on $X$. Of course, we would like to see $X/G$ again as an algebraic variety, but this cannot work in general, because $X$ usually contains non-closed orbits, and so $X/G$ contains non-closed points.

If $X$ is an affine variety with coordinate ring $\OOO(X)$, we could look at the subalgebra $\OOO(X)^{G}\subset \OOO(X)$ of $G$-invariant functions and consider the morphism
$$
\pi_{X}\colon X \to X\quot G := \Spec\OOO(X)^{G}
$$
induced by the inclusion. It is called {\it algebraic quotient\/}, is a categorical quotient in the category of affine schemes, and so has the usual universal property:  {\it Every $G$-invariant morphism $X \to Y$ factors uniquely through $\pi_{X}$}. In some sense this is the best algebraic approximation to the orbit space. 

If $G$ is reductive, then $\OOO(X)^{G}$ is finitely generated and $\pi_{X}$ has some very nice properties (see \cite[II.3.2]{Kr1984Geometrische-Metho}):
\be
\bullitem $\pi_{X}$ is {\it $G$-closed:} If $Z\subset X$ is $G$-stable and closed, then $\pi_{X}(Z)$ is closed. 
\bullitem $\pi_{X}$ is {\it $G$-separating:} If $Z,Z' \subset X$ are disjoint $G$-stable closed subsets, then $\pi_{X}(Z)\cap\pi_{X}(Z') = \emptyset$.
\ee
In particular, $\pi_{X}$ is surjective and every fiber contains a unique closed orbit. Thus $X \quot G$ classifies the closed orbits in $X$. In good situations, the generic orbits are closed, and so at least generically $X\quot G$ is the orbit space.

If $G$ is not reductive, then all this fails to be true. In particular, the invariant ring might not be finitely generated and so the quotient $X\quot G$ is not algebraic, and the quotient morphism $\pi_{X}$ is not usually surjective. The fact that $X \quot G$ is not algebraic was considered to be the main difficulty in handling non-reductive groups. We think that the non-surjectivity of $\pi_{X}$ is an even a more serious problem.

One of the aims of this paper is to show that the quotient $X\quot G$ as a scheme is ``almost'' algebraic. It always contains large open algebraic subsets and shares many properties with algebraic varieties. This is explained in  sections~\reff{general.sec} and \reff{sep.sec} which are inspired by \name{Winkelmann}'s work \cite{Wi2003Invariant-rings-an}. For example, if the base field is uncountable, then $X\quot G$ is a \name{Jacobson} scheme which implies that the \name{Zariski} topology on $X\quot G$ is determined by the \name{Zariski} topology on the $\k$-rational points of $X\quot G$. 

To have an idea of our approach and our results let us give a geometric interpretation of \name{Roberts} famous example of an action of the additive group $\Ga = (\k,+)$ on $\AA^{7}$ with a non-finitely generated ring of invariants  (see section~\reff{Roberts.sec} for details). Let $\pi\colon \AA^{7} \to \AA^{7}\quot \Ga$ be the quotient. Then
{\it
\be
\item 
The fixed point set $F:=(\AA^{7})^{\Ga}\simeq\AA^{4}$ is mapped to a single point $\pi(0)$;
\item  
The complement $\AA^{7}_{\bd} := \AA^{7}\setminus F$ is a principal $\Ga$-bundle  $\pi\colon \AA^{7}_{\bd} \to \pi(\AA^{7}_{\bd})$;
\item 
The image $\pi(\AA^{7}_{\bd}) \subset \AA^{7}\quot \Ga$ is an open algebraic subset and contains every open algebraic subset $U$ of  $\AA^{7}\quot \Ga$;
\item 
The complement $\AA^{7}\setminus\pi(\AA^{7}_{\bd})$ is isomorphic to $\AA^{3}$.
\ee
}
An important new feature is the concept of {\it separating morphisms} $\phi\colon X \to Y$ where $Y$ is an algebraic variety (cf. \cite[section~2.3]{DeKe2002Computational-inva}). This means that $\phi$ is $G$ invariant and separates the same orbits as $\pi_{X}$. Such morphisms always exist even when the invariants are not finitely generated, but finding a ``nice'' separating morphism is usually a difficult task. For \name{Roberts} example we get the following.
{\it
\be
\item[(e)]
There exists a separating morphism $\phi\colon \AA^{7}\to \AA^{9}$ such that $Y:=\overline{\phi(\AA^{7})}$ is normal of dimension 6.
\item[(f)]
The induced map $\bar\phi\colon \AA^{7}\quot \Ga \to Y$ gives a homeomorphism $\pi(\AA^{7}) \to \phi(\AA^{7})$ and an isomorphism $\pi(\AA^{7}_{\bd}) \simto \phi(\AA^{7}_{\bd})$.
\item[(g)]
$H:=Y\setminus\phi(\AA^{7}_{\bd})$ is a hypersurface in $Y$, and $\OOO(\AA^{7})^{\Ga}=\OOO_{Y}(Y\setminus H)$.
\ee
}
Another important concept is the {\it separating variety\/} which measures how much the invariants separate the orbits. It is defined as the reduced fiber product $\SSS_{X}:=X \times_{X\quot G}X$ and contains the closure of the graph $\Gamma_{X}:=\{(gx,x)\mid g\in G, x \in X\}$. If a general fiber of the quotient map is an orbit and $G$ is connected, then $\overline{\Gamma_X}$ is an irreducible component of the separating variety. But even in nice situations, the separating variety may have additional components. In general, the meaning of the other components is not yet well understood, except for some special cases (see below). For \name{Roberts} example we find
{\it
\be
\item[(h)] The separating variety has two irreducible components: $\SSS_{\AA^{7}}= \overline{\Gamma_{\AA^{7}}}\cup F \times F$.
\ee
}
The most complete results are obtained for actions of the additive group $\Ga$, in particular for representations of $\Ga$, see sections~\reff{Ga-actions.sec}--\reff{sepvar.sec}. This part of our work was inspired by certain calculations done by \name{Elmer} and \name{Kohls} in \cite{ElKo2012Separating-invaria}. An important tool is the geometric interpretation of the zero set of the {\it plinth ideal\/}. If $X$ is factorial, it is the complement of the open set $X_{\bd}$ where $X$ is locally a $\Ga$-bundle.
In section~\reff{Ga-SL2.sec} we generalize some of the results for representations to $\Ga$-actions induced by actions of $\SLtwo$.

Finally, the last section contains a detailed analysis of \name{Roberts}' example of a $\Ga$-action on $\AA^{7}$ with a non-finitely generated ring of invariants. To prepare the reader for the difficulties in working with non-finitely generated algebras we describe an easy example in section~\reff{Example1.sec}.

\par\bigskip
\section{General Setup and Notation}\lab{general.sec}
\subsection*{Invariants}\lab{invariants.subsec}
Our base field $\k$ is algebraically closed. In the second part, starting with sections~\reff{Ga-actions.sec}, we study $\Ga$-actions and will assume that $\Char\k=0$. Since we have to deal with non-finitely generated rings of invariants, we will work in the category of {\it $\k$-schemes $X$}, and will denote by $X(\k)$ the {\it $\k$-rational points of $X$}. In this setting, a {\it variety $X$}  is a reduced algebraic $\k$-scheme. For a variety $X$, we will often confuse the scheme $X$ with its $\k$-rational points $X(\k)$.

Throughout this paper, we let $X$ be a normal affine variety and $G$ an algebraic group acting on $X$. We denote by $\OOO(X)$ the $\k$-algebra of regular functions on $X$ and by $\OOO(X)^{G} \subset \OOO(X)$ the subalgebra of $G$-invariant functions. The {\it quotient\/} is define to be the affine $\k$-scheme
$$
X\quot G := \Spec\OOO(X)^{G}.
$$ 
If the base field $\k$ is uncountable, a famous result of \name{Krull}'s implies that $X \quot G$ is a {\it \name{Jacobson} scheme\/}, i.e., $\OOO(X)^{G}$ is a {\it \name{Jacobson} ring} (\cite{Kr1951Jacobsonsche-Ringe}). This means that every radical ideal of $\OOO(X)^{G}$ is the intersection of maximal ideals. Moreover, every closed point of $X$ is $\k$-rational in this case, since $\OOO(X)^{G}$ is contained in a finitely generated $\k$-algebra. It follows that  the \name{Zariski}-topology on $X\quot G$ is completely determined by the \name{Zariski}-topology on the $\k$-rational points $(X\quot G)(\k)$. This allows to work with $\k$-rational points which are the only interesting objects from a geometric point of view.

\begin{rem} 
If the $\k$-algebra $R$ is not a \name{Jacobson} ring, then there is a prime ideal $\pp \subset R$ which is not the intersection of the maximal ideal containing $\pp$. In geometric terms this means the following. Denote by $Z \subset \Spec R$ the closed subscheme defined by $\pp$, and let $Z_{\text{\cl}} \subset Z$ be the subset of closed points. Then the closure $\overline{Z_{\text{\cl}}}$ in $\Spec R$ is strictly contained in $Z$.
\end{rem} 

\subsection*{Quotient morphism}\lab{quotient.subsec}
The inclusion $\OOO(X)^{G}\into\OOO(X)$ defines the {\it quotient morphism}
$$
\pi=\pi_{X}\colon X \to X\quot G.
$$
Although $\OOO(X)^{G}$ might not be finitely generated over $\k$ (hence $X\quot G$ is not algebraic), we will see that the quotient $X\quot G$ contains large open sets which are algebraic. For this we need the following result due to \name{Derksen} and \name{Kemper} \cite[Proposition 2.9 and 2.7]{DeKe2008Computing-invarian}.
\begin{prop}
Let $R$ be a $\k$-algebra. Define
$$
\ff_{R}:=\{ f \in R  \mid R_{f} \text{ is finitely generated}\} \cup \{0\}.
$$
Then $\ff_{R}$ is a radical ideal of $R$. If $R$ is contained in a finitely generated $\k$-domain,
then $\ff_{R} \neq (0)$.
\end{prop}
The ideal $\ff_{R}$ will be called the  {\it finite generation ideal}. 

\begin{rem}\lab{Jacobson.rem}
The open subset $\Spec R \setminus \VVV(\ff_{R})\subset \Spec R$ is the union of all open subschemes $U \subset \Spec R$ which are algebraic.
In fact, each such $U$ is a finite union of open affine algebraic $U_{i}$, and each $U_{i}$ is a finite union of some $(\Spec R)_{f_{j}}$.
We will denote $\Spec R \setminus \VVV(\ff_{R})$ by $(\Spec R)_{\alg}$ and call it the {\it algebraic locus}:
$$
(\Spec R)_{\alg}: = \Spec R \setminus \VVV(\ff_{R}) =  \bigcup_{\substack{U \subset \Spec R\\\text{ open and algebraic}}} U
$$
Note that $(\Spec R)_{\alg}$ is itself algebraic if and only if $\ff_{R}$ is the radical of a finitely generated ideal. On the other hand, $(\Spec R)_{\alg}$ is always Jacobson and its closed points coincide with its $\k$-rational points.
\end{rem}

\begin{defn}
Let $Z = \Spec R$ be an affine $\k$-scheme.
If $A \subset Z$ is a closed subset we define $I(A) \subset R$ to be the (radical) ideal of functions vanishing on $A$.
\be
\item $\dim Z := \Kdim R$, the \name{Krull}-dimension of $R$.
\item If $Z$ is reduced and irreducible, i.e., if $R$ is a domain, then $\k(Z):=Q(R)$ denotes  the field of fractions of $R$.
\item If $R$ is a domain, then $\tdeg_{\k} R:=\tdeg_{\k}Q(R)$ is the transcendence degree of the field extension  $Q(R)/\k$.
\item If $A \subset Z$ is closed, then $\codim_{Z} A :=\min\{\height\pp\mid\pp \supset I(A), \pp\text{ prime}\}$ where $\height\pp$ is the height of the prime ideal $\pp$.
\ee
\end{defn}
As an example, we will see later in Theorem~\reff{genthm}(a) that the quotient $X\quot G$ introduced above is always finite dimensional, and that  $\dim X\quot G = \tdeg_{\k}\OOO(X)^{G}$.

\subsection*{Algebraic varieties}\lab{algebraic.subsec}
Assume that $Z=\Spec R$ is algebraic. Then $Z=\bigcup_{i}Z_{i}$ is a finite union of irreducible closed subsets, and $\dim Z = \max_{i}\{\dim Z_{i}\}$. Moreover, if $Z$ is  reduced and irreducible, then $\dim Z = \tdeg_{\k} R$, and for every irreducible closed subset $A \subset Z$ we have $\dim A + \codim_{Z} A = \dim Z$. 

Finally, if $\phi\colon Z \to Y$ is a morphism where $Y$ is an arbitrary $\k$-scheme, and if $A\subset Z$ is a closed subscheme, then $\phi(A(\k))$ is dense in $\phi(A)\subset Y$. As mentioned before, this last statement holds more generally if $R$ is a \name{Jacobson} ring.

\bigskip
\section{A First Example}\lab{Example1.sec}
Let us discuss an interesting example. While it does not quite fit in our setting---it does not arise from a quotient of an algebraic group action on a normal affine variety---it has a similar behavior.

Consider the graded subring $R:=\k[x,xy,xy^{2},xy^{3},\ldots] \subset \k[x,y]$ generated by the monomials $xy^{k}$, $k=0,1,\ldots$, and set $Z := \Spec R$.

{\it 
\be
\item The finite generation ideal  $\ff_{R}$  of $R$ is equal to  the homogeneous maximal ideal $\mm_{0} = (x,xy,xy^{2},\ldots)$, and  $\mm_{0} = \sqrt{xR}$.
\item We have $Z\setminus \{\mm_{0}\} = Z_{x}$, and this is an affine algebraic variety with coordinate ring $\k[x,x^{-1},y]$
\ee
}
Now consider the morphism $\pi\colon \Atwo \to Z$ given by the inclusion $R \subset \k[x,y]$. (This morphism will play the role of a quotient morphism.)
{\it
\be
\item[(c)] $\pi\colon\Atwo \to Z$ is surjective and induces an isomorphism $(\Atwo)_{x}\simto Z_{x}$.
\item[(d)] $\pi\colon\Atwo \to Z$ is a closed morphism.
\ee
}
Finally, we consider the affine morphism $\phi\colon\Atwo \to \Atwo$ given by $(x,y)\mapsto (x,xy)$.
{\it
\be
\item[(e)] $\phi$ factors through $\pi$
\begin{diagram}
\Atwo & \rTo^{\pi}& Z\\
& \rdTo_{\phi} & \dTo_{\bar\phi} \\
&& \Atwo
\end{diagram}
and $\bar\phi$ is injective on the image of $\pi$. Hence $\phi$ separates the same points of $\Atwo$ as $\pi$.
\item[(f)] $\bar\phi$ induces a homeomorphism $Z \to \phi(\Atwo) = \Atwo_{y}\cup\{0\}$. 
\ee
}
The proofs are not difficult and are left to the reader. They  are based on the following lemma.
\begin{lem}
\be
\item We have $R = \k \oplus \mm_{0}$ where $\mm_{0}=x\k[x,y]=(x,xy,xy^{2},\ldots)$ is the homogeneous maximal ideal of $R$.
\item Let $f\in \k[x,y]$. Then
$$
f\k[x,y]  \cap R =
\begin{cases}
f\k[x,y] & \text{if } f\in\mm_{0};\\
fR & \text{if }f\in R\setminus \mm_{0};\\
(xf)R & \text{if }f\notin R.
\end{cases}
$$
\ee
\end{lem}

\bigskip
\section{Separating Morphisms}\lab{sep.sec}
\subsection*{Separation}\lab{separation.subsec}
The so-called {\it separation property\/} will play an important role in this paper. The notion goes back to \name{Derksen} and \name{Kemper}  \cite[section~2.3.2]{DeKe2002Computational-inva}, and is also implicit in the work of \name{Winkelmann} \cite[Lemma 7]{Wi2003Invariant-rings-an}.

\begin{defn} Let $X$ be an affine $G$-variety.
A $G$-invariant morphism $\phi\colon X \to Y$ where $Y$ is an affine variety is a {\it separating morphism\/} if it satisfies the following {\it Separation Property:}
\be
\item[(SP)] {\it If $x,x' \in X(\kk)$ are separated by an invariant $f\in\OOO(X)^{G}$, i.e., if $f(x)\neq f(x')$, then $\phi(x)\neq \phi(x')$.}
\ee
\end{defn}

\begin{rem} 
If $\Char\k = 0$, then the separation property (SP) implies that $\phi^{*}$ induces an isomorphism $\k(\overline{\phi(X)})\simto Q(\OOO(X)^{G})$. If $\Char\k>0$, we say that $\phi$ is {\it strongly separating\/} if $\phi$ is separating and induces an isomorphism $\k(\overline{\phi(X)}) \simto Q(\OOO(X)^{G})$.
\end{rem}

It is shown in \cite[Theorem~2.3.15]{DeKe2002Computational-inva} that separating morphisms always exist. In more algebraic terms this means that one can find a finitely generated {\it separating subalgebra\/} $R \subset \OOO(X)^{G}$, i.e., a subalgebra which separates the same $\k$-rational points of $X$ as the invariant functions. We can always add invariant functions to $R$, and thus assume that $R$ is normal and that $Q(R)=Q(\OOO(X)^{G})$, if necessary. Thus, a strongly separating morphism $\phi\colon X\to Y$ with $Y$ normal always exists. A basic problem is to find a separating algebra with a small number of generators.

\subsection*{Main results}\lab{main.subsec}
A $G$-invariant morphism $\phi\colon X \to Y$ always factors through the quotient morphism $\pi\colon X \to X\quot G$:
\begin{diagram}
X & \rTo^{\pi}& X\quot G\\
& \rdTo_{\phi} & \dTo_{\bar\phi} \\
&& Y 
\end{diagram}
Then $\phi$ is separating if and only if  $\bar\phi$ is injective on the image $\pi(X(\kk)) \subset (X\quot G)(\k)$ of the rational points. In the paper \cite{Wi2003Invariant-rings-an}, \name{Winkelmann} studies this general set-up and proves a number of fundamental results, e.g. that every such invariant ring $\OOO(X)^{G}$  is the ring of global regular functions on a quasi-affine variety and vice versa. Some of his results are contained and extended in the following theorem, where we take a geometric point of view.

\begin{thm}\lab{genthm}
Let $X$ be a normal affine variety with an action of an algebraic group $G$, and denote by $\pi\colon X \to X\quot G$ the quotient morphism. Let $\phi\colon X \to Y$ be a dominant separating morphism where $Y$ is a normal affine variety.
\be
\item 
If $A \subset X$ is an irreducible closed subset, then $\dim\overline{\pi(A)} = \dim\overline{\phi(A)}$ and $\codim_{X\quot G}\overline{\pi(A)} = \codim_{Y}\overline{\phi(A)}$. In particular, 
$$
\dim X\quot G = \dim Y = \tdeg_{\k}\OOO(X)^{G}.
$$
\item 
The map $\bar\phi\colon X\quot G \to Y$ induces a homeomorphism $\pi(X) \simto \phi(X)$.
\item 
We always have $\codim_{X\quot G}\overline{X\quot G\setminus \pi(X)} > 1$.
\ee
For the next four statements we assume that $\phi$ is strongly separating.
\be\setcounter{enumi}{3}
\item Set $N:= Y\setminus\phi(X)$. Then $\OOO(X)^{G} = \OOO(Y\setminus \overline{N})$. 
\item
$\bar\phi^{-1}(Y \setminus\overline{N})\subset (X\quot G)_{\alg}$ and the induced map $\bar\phi^{-1}(Y \setminus\overline{N}) \simto Y \setminus\overline{N}$ is an isomorphism. 
\item
The complement $X\quot G \setminus (X\quot G)_{\alg}$ has codimension $>1$ in $X\quot G$.
\item
Set $M:=X\quot G \setminus \pi(X) \subset X\quot G$. Then
$\bar\phi$ induces an open immersion $(X\quot G)_{\alg}\setminus \overline{M} \into Y$.
\ee
\end{thm}
The proofs of the following corollaries are easy and left to the reader.
\begin{cor}
Assume $\phi\colon X\to Y$ is dominant and strongly separating with $Y$ normal. If the finite generation ideal $\ff_{X\quot G}$ of $\OOO(X)^G$ is the radical of the ideal generated by $\ff_{X\quot G}\cap\OOO(Y)$, then $(X\quot G)_{\alg}$ is algebraic and $\bar\phi$ induces an open immersion $(X\quot G)_{\alg} \into Y$.
\end{cor}
\begin{cor}
Assume $\phi\colon X\to Y$ is dominant and strongly separating. If $Y$ is factorial, then $\OOO(X)^{G}$ is finitely generated and $\bar\phi\colon X\quot G \to Y$ is an open immersion. In particular, $X\quot G \simeq Y_{f}$ for a suitable $f\in\OOO(Y)$.
\end{cor}
\begin{cor} 
If $V$ is a rational representation of an algebraic group $G$ and if $\phi\colon X\to Y$ is a strongly separating morphism with $Y$ factorial, then $Y=X\quot G$. 
\end{cor}
\begin{rem}
In the case where $G$ is reductive, this last corollary is an easy consequence of \name{Richardson}'s Lemma (see \cite[II.3.4]{Kr1984Geometrische-Metho}).
\end{rem}

We say that an affine $\k$-scheme $Z=\Spec R$ is {\it fix-pointed} with fixed point $z_{0}$, if $R = \bigoplus_{i\geq 0}R_{i}$ is a graded ring with $R_{0}=\k$ and $z_{0}$ the homogeneous maximal ideal. Geometrically this means that $Z$ admits an action of the {\it multiplicative group\/} $\Gm:=\kst$ with a single closed orbit, namely the fixed point $z_{0}$. A variety $X$ is called a {\it fix-pointed\/} $G$-variety if $X$ is fix-pointed and the  $G$-action commutes with the $\Gm$-action. In this case $X\quot G$ is also fix-pointed, and the finite generation ideal $\ff_{X\quot G}$ is homogeneous.

\begin{cor}\lab{fixpointed.cor}
Let $(X,x_{0})$ be a fix-pointed affine $G$-variety. If $\pi(x_{0})\notin \overline{X\quot G \setminus \pi(X)}$, then $\OOO(X)^{G}$ is finitely generated, $\pi$ is surjective and $\bar\phi\colon X\quot G \into Y$ an open immersion. If, in addition, $Y$ is also fix-pointed and $\phi$ is homogeneous, then $\bar\phi\colon X\quot G \simto Y$ is an isomorphism.
\end{cor}
Note that the special case of Corollary~\reff{fixpointed.cor} for a representation of a reductive group $G$ is contained in \cite[Proposition~2.3.12]{DeKe2002Computational-inva}.

\subsection*{Proof of Theorem~\ref{genthm}}\lab{proof1.subsec}
The proof needs some preparation.
\begin{lem}\lab{lem1}
Let $W$ be an irreducible affine variety, $R \subset \OOO(W)$ a $\k$-subalgebra and $\psi\colon W \to Z:=\Spec R$ the induced morphism. Then there is an $f\in\ff_{R}$ and a finite surjective morphism $\rho\colon W_{f}\to Z_{f}\times \k^{m}$, where $m:=\dim W - \tdeg_{\k}Q(R)$,  such that $\psi|_{W_{f}}=\pr_{Z_{f}}\circ\rho$:
\begin{diagram}
W_{f} & \rTo^{\rho}& Z_{f}\times \k^{m}\\
& \rdTo_{\psi} & \dTo_{\pr_{Z_{f}}} \\
&& Z_{f} 
\end{diagram}
In particular, there is a subset $U \subset \psi(W)$ which is open, algebraic and dense in $Z$.
\end{lem}
\begin{proof} By first inverting some $f\in\ff_{R}$ we can assume that $R$ is finitely generated. In this case the result is known and can be found in 
\cite[Chap. V.3.1, Corollary~1]{Bo1998Commutative-algebr}, cf. \cite[Appendix A.3.4 Decomposition Theorem]{Kr2011Algebraic-Transfor}.
\end{proof}
The following two results can be found in \cite[Lemma 1, 2, and~6]{Wi2003Invariant-rings-an}. The first is due to \name{Nagata} \cite{Na1965Lectures-on-the-fo}. 
\begin{lem}\lab{lem2}
The invariant ring $R:=\OOO(X)^{G}$ is a \name{Krull}-ring, i.e., $R = \bigcap_{\pp} R_{\pp}$ where $\pp$ runs through the primes of $R$ of height 1.
\end{lem}

\begin{lem}\lab{lem3}
Let $S \subset X\quot G$ be an irreducible closed subscheme of codimension 1, and put $H:=\pi^{-1}(S)
\subset X$. Then $S=\overline{\pi(H)}$.
\end{lem}
We will also need the following result; the proof is easy and left to the reader.
\begin{lem}\lab{lem4}
Let $Y$ be an irreducible variety, $C \subset Y$ an irreducible closed subset of codimension $d$ and $U \subset Y$ a non-empty open set. 
Then there is a chain 
$$
Y=C_{0}\supset C_{1}\supset \cdots\supset C_{d}=C
$$ 
of closed irreducible subsets such that 
\be
\item[(i)] $\codim_{Y}C_{j}= j$ for $j=0,\ldots,d$, and 
\item[(ii)] $C_{j}\cap U\neq\emptyset$ for $j<d$.
\ee
\end{lem}

\begin{proof}[Proof of Theorem~\ref{genthm}]
(a) Lemma~\reff{lem1} implies that there is an open set $U \subset A$ such that $\pi(U)$ is open, algebraic and dense in $\overline{\pi(A)}$, 
and that $\phi(U)$ is open, algebraic and dense in $\overline{\phi(A)}$. Now $\pi(U(\kk)) \to \phi(U(\kk))$ is bijective, since $\phi$ is separating. As $\pi(U)$ and $\phi(U)$ are algebraic, it follows that $\dim\overline{\pi(A)} = \dim \pi(U) = 
\dim \phi(U) = \dim\overline{\phi(A)}$. 

To get the equality for the codimensions, we choose a non-empty open subset $O\subset X$ such that $\pi(O)$ is open and algebraic in $X\quot G$, and such that $U:=\phi(O)$ is open in $Y$. From Lemma~\reff{lem4} there is a sequence $C_{0}=Y \supset C_{1}\supset\cdots\supset C_{d}=\overline{\phi(A)}$ of closed irreducible subsets $C_{j}$ with $\dim C_{j}=\dim Y-j$ such that $C_{j}\cap U\neq \emptyset$ for $j<d$. Since $\bar\phi\colon \pi(O) \to \phi(O)$ is a bijective morphism of varieties, we see that, for   $j<d$, $B_{j}:=\overline{\bar\phi^{-1}(C_{j}) \cap \pi(O)}$ is irreducible of dimension $\dim Y-j$, and that $B_{j}\subset B_{j-1}$.  
It remains to see that $B_{d-1} \supseteq \pi(A)$, since this implies that $\codim_{X\quot G} \overline{\pi(A)} \geq d = \codim_{Y}\overline{\phi(A)}$. If not, using again Lemma~\reff{lem1}, we can find a subset $U \subset \pi(A)$ which is open and dense in $\overline{\pi(A)}$ and such that $U \cap B_{d-1}=\emptyset$. 
Then the image $\bar\phi(U)$ is disjoint from $\bar\phi(B_{d-1}\cap \pi(O))$. Since $\overline{\bar\phi(B_{d-1}\cap \pi(O))} = C_{d-1}$, it follows that $\overline{\phi(A)} = \overline{\bar\phi(U)}$ is not contained in $C_{d-1}$, contradicting the assumption.
\par\smallskip
(b) The same argument as above shows that, for irreducible closed subsets $A,B \subset X$ with $\overline{\pi(A)}\nsubseteq\overline{\pi(B)}$, we have $\overline{\phi(A)}\nsubseteq\overline{\phi(B)}$. It follows that the map $\pi(X) \to \phi(X)$ is injective, hence bijective, and open, hence a homeomorphism.
\par\smallskip
(c) For $\pp\in X\quot G$ we have $\pp\in M:=(X\quot G)\setminus \pi(X)$ if and only if $\overline{\pi(\VVV_{X}(\pp))}\subsetneqq \VVV(\pp)$ where $\VVV(\pp)$ denotes the zero set in $X\quot G$.
Assume now that $\codim_{X\quot G}\overline{M} =1$. This means that $\overline{M}$ contains an irreducible closed subscheme $S$ of codimension 1 corresponding to a prime ideal $\pp\in M$ of height 1. It follows that $\overline{\pi(\pi^{-1}(S))}\subsetneqq S$, contradicting Lemma~\reff{lem3}.
\par\smallskip
(d) Let $S\subset X\quot G$ be an irreducible hypersurface and let $\pp \subset R:=\OOO(X)^{G}$ be the corresponding prime ideal of height 1. Then, by Lemma~\reff{lem3} and (b),  $H:=\overline{\bar\phi(S)}$ is an irreducible hypersurface, and so the corresponding prime ideal $\pp':=\pp\cap\OOO(y)$ has also height 1. This implies that $\OOO(Y)_{\pp'}=R_{\pp}$, since both are discrete valuation rings of $\k(Y)$. But every irreducible hypersurface $H \subset Y$ not contained in $\overline{N}$ is of the form $\overline{\bar\phi(S)}$, hence $\OOO(Y\setminus\overline{N}) = \bigcap_{\pp'=\pp\cap\OOO(Y)}\OOO(Y)_{\pp'} = \bigcap_{\pp}R_{\pp} = R$ by Lemma~\reff{lem2}.
\par\smallskip
(e) If $f\in I(\overline{N})$, then $Y_{f}\subset Y \setminus\overline{N}$, and so $\OOO(Y)_{f} \supset \OOO(Y\setminus \overline{N})=\OOO(X)^{G}$ by (d). Thus $\bar\phi$ induces an isomorphism $(X\quot G)_{f} \simeq Y_{f}$, and so $(X\quot G)_{f}$ is algebraic.
\par\smallskip
(f) By construction, $\overline{\phi(\phi^{-1}(\overline{N}))}$ does not contain a hyperplane, and  neither does $\bar\phi^{-1}(\overline{N})$ by Lemma~\reff{lem3} and (a). The claim now follows since $\bar\phi^{-1}(\overline{N})\supset X\quot G \setminus (X\quot G)_{\alg}$, as we have just seen in (e).
\par\smallskip
(g) By (b), $\bar\phi\colon X\quot G \setminus \overline{M} \to Y$ is injective. Hence, for every open algebraic  $U \subset X\quot G$, the map $\bar\phi\colon U\setminus \overline{M} \to Y$ is an open immersion by \name{Zariski}'s Main Theorem \cite[Th\'eor\`eme~8.12.6]{Gr1967Elements-de-geomet}. The claim then follows.
\end{proof}

\bigskip\bigskip
\section{$\Ga$-Actions, Local Slices, and the Plinth Variety}\lab{Ga-actions.sec}
\subsection*{$\Ga$-bundles}\lab{GaBundles.subsec}
From now on we assume that $\Char\k =0$.
In this and the following sections we focus on $\Ga$-varieties, i.e., varieties with an action of the additive group $\Ga\simeq (\k,+)$.
A $\Ga$-variety $X$ (not necessarily affine) is called a \emph{trivial $\Ga$-bundle} if there is a $\Ga$-equivariant isomorphism $\Ga \times Y \simto X$, or, equivalently, if there is a $\Ga$-equivariant morphism $X \to \Ga$. In this case, $Y$ can be identified with the orbit space $X/\Ga$, and the quotient morphism $\pi\colon X \to X/\Ga$ admits a section. If $X$ is affine, then $X/\Ga = \Spec \OOO(X)^{\Ga}$.

The $\Ga$-variety $X$ is called a \emph{principal $\Ga$-bundle} (for short, a $\Ga$-bundle) if there is a $\Ga$-invariant morphism $\pi\colon X \to Z$ and an open covering $Z = \bigcup_{i} U_{i}$ such that $p^{-1}(U_{i})\to U_{i}$ is a trivial $\Ga$-bundle for all $i$. In this case, $Z$ can be identified with the orbit space $X/\Ga$ and the morphism $\pi$ has the usual universal properties. Again, if $X$ is affine, then $X/\Ga = \Spec\OOO(X)^{\Ga}$.

\subsection*{Local slices}\lab{slice.subsec}
Now let $X$ be a normal affine $\Ga$-variety. The $\Ga$-action defines a locally nilpotent vector field $D\in\VF(X):=\Der_{\k}(\OOO(X))$  which determines the $\Ga$-action. Its kernel coincides with the ring of invariants: $\ker D = \OOO(X)^{\Ga}$. If $s \in \OOO(X)^{\Ga}$ is a non-zero invariant and $s = Df$ for some $f\in\OOO(X)$, then $D(\frac{f}{s}) = 1$ and thus the morphism 
$$
\frac{f}{s} \colon X_{s} \to \Ga
$$ 
is $\Ga$-equivariant. Such morphisms are called \emph{local slices}. It follows that the affine open set $X_{s}$ is a trivial $\Ga$-bundle, and $X_{s}/\Ga = \Spec\OOO(X_{s})^{\Ga}$. In particular, $\OOO(X_{s})^{\Ga}=(\OOO(X)^{\Ga})_{s}$ is finitely generated. 

\begin{defn} \lab{plinth.def}
Let $X$ be a normal affine $\Ga$-variety. The ideal $\pp_{X} \subset \OOO(X)^{\Ga}$ generated by all $s\in\OOO(X)^{\Ga}$ of the form $s = Df$ for some $f\in\OOO(X)$ is called the \emph{plinth ideal}:
$$
\pp_{X}:= D(\OOO(X)) \cap \ker D \subset \OOO(X)^{\Ga}.
$$
The zero set  $\PPP_{X}:=\VVV_{X}(\pp_{X})\subset X$ of the plinth ideal is called the \emph{plinth variety\/} of $X$.
Note that the plinth ideal is an ideal in the invariant ring, whereas the plinth variety is a closed subvariety of $X$.
\end{defn}

As before, the quotient morphism is denote by $\pi\colon X\to X\quot\Ga$.
The next result shows that outside the plinth variety the quotient morphism is a principal bundle.

\begin{prop}\lab{plinth.prop}
The image $\pi(X \setminus \PPP_{X})\subset X\quot \Ga$ is an open algebraic variety, and the morphism $\pi\colon X \setminus \PPP_{X}\to \pi(X \setminus \PPP_{X})$ is a principal $\Ga$-bundle.
\end{prop}
\begin{proof} If $s=Df$ and $Ds=0$, then $\pi(X_{s}) = (X\quot\Ga)_{s}$, and this is an open subset of $X\quot\Ga$ which is affine and algebraic. Since we can cover $X\setminus \PPP_{X}$ with finitely many $X_{s_{j}}$ we see that $\pi(X\setminus \PPP_{X})$ is covered by finitely many open affine varieties, hence is an algebraic variety. It remains to see that $\pi$ separates the $\Ga$-orbits on $X\setminus \PPP_{X}$. This is clear for two orbits contained in the same $X_{s_{j}}$. If $O_{1}\subset X_{s_{j}}$ and $O_{2}\subset X_{s_{k}}\setminus X_{s_{j}}$, then the invariant $s_{j}$ vanishes on $O_{2}$, but not on $O_{1}$. 
\end{proof}

\begin{defn} Let $X$ be a $\Ga$-variety. Define $\Xbd \subset X$ to be the union of all open $\Ga$-stable subsets $U$ which are trivial $\Ga$-bundles:
$$
\Xbd := \bigcup_{\substack{U\subset X \text{ open} \\ U\text{ a trivial $\Ga$-bundle}}} U.
$$
\end{defn}
If $X$ is affine, it follows from Proposition~\reff{plinth.prop} that $X \setminus\PPP_{X}\subset \Xbd$. We will see later (Example~\reff{SLtwoT.exa}) that the inclusion can be strict. However, this cannot happen if $X$ is factorial.
\begin{prop}
Let $X$ be a factorial affine $\Ga$-variety. Then
$$
\Xbd = X \setminus \PPP_{X}.
$$
In particular, $\pi(\Xbd) \subset X\quot\Ga$ is open and algebraic and  $\Xbd\to\pi(\Xbd)$ is a principal $\Ga$-bundle.
\end{prop} 
\begin{proof} In the definition of $\Xbd$ we can assume that all $U_{i}$ are affine. Since $X$ is factorial, this implies that $U_{i}= X_{t_{i}}$ for a suitable invariant $t_{i}$. On the other hand, if $X_{t}$ is a trivial $\Ga$-bundle where $t\in\OOO(X)^{\Ga}$, then there is an $h\in\OOO(X_{t})$ such that $Dh = 1$. Writing $h = f t^{-k}$ we see that $s:=t^{k}=Df$, and so $X_{s}= X_{t}$ is of the form above. 
\end{proof}

\par\bigskip
\section{The case of a representation}\lab{rep.sec}
\subsection*{Representations and the null cone}\lab{rep-nullcone.subsec}
Let $V$ be representation of $\Ga$. Then $V$ extends to a representation of $\SLtwo:=\SLtwo(\k)$, where $\Ga$ is identified with the unipotent subgroup $U\subset\SLtwo$ via $s \mapsto \begin{bmatrix} 1 & s \\ 0 & 1 \end{bmatrix}$.  The invariants $\OOO(V)^{\Ga}$ are finitely generated (\name{Weitzenb\"ock}'s Theorem, see \cite[III.3.9]{Kr1984Geometrische-Metho}), and the multiplicative group $\Gm$ acts linearly on $V$, $(t,v)\mapsto t\cdot v$, via the identification $t \mapsto \left[\begin{smallmatrix} t&\\&t^{-1}\end{smallmatrix}\right]\in T \subset \SLtwo$. This defines a decomposition of $V$ into weight spaces:
$$ 
V = \bigoplus_{k}V_{k}, \quad V_{k}:=\{ v\in V \mid t\cdot v = t^{k} v\}.
$$ 

Since the invariants are finitely generated, the quotient $V\quot \Ga:=\Spec\OOO(X)^{\Ga}$ is an affine variety. As usual, the \emph{nullcone} is defined by $\NNN=\NNN_{V}:=\pi^{-1}(\pi(0)) \subset V$. Recall that the \name{Weyl}-group $W\simeq \ZZ/2\ZZ$ of $\SLtwo$ acts on the zero weight space $V_{0}=V^{\Gm}$. The non-trivial element of $W$ is represented by the matrix $\sigma=\left[\begin{smallmatrix} 0&-1\\1&0\end{smallmatrix}\right]\in\SLtwo$.

\begin{thm}\lab{NV.thm}
\be
\item
$\NNN_{V}= V^{+}:=\bigoplus_{k>0}V_{k}$.
\item 
$\PPP_{V}=V \setminus \Vbd = V_{0}\oplus V^{+}$. In particular, $\PPP_{V}=\NNN_{V}$ if and only if the $\SLtwo$-representation $V$ does not contain odd-dimensional irreducible representations.
\item
The image $\pi(\PPP_{V}) \subset V\quot\Ga$ is closed. The induced map
$\pi|_{\PPP_{V}}\colon \PPP_{V}\to \pi(\PPP_{V})$ is given by the $\SLtwo$-invariants and has a factorization
$$
\begin{CD}
\PPP_{V}=V^{+}\oplus V_{0} @>{\pr}>> V_{0} @>{\pi_{0}}>> V_{0}/W @>{\bar\pi}>> \pi(\PPP_{V})
\end{CD}
$$
where $\pi_{0}$ is the quotient by $W$ and $\bar\pi$ is finite and bijective.
\ee
\end{thm}
\begin{rem}
\name{Elmer} and \name{Kohls} \cite{ElKo2012Separating-invaria} gave an explicit construction of separating sets for indecomposable representations, which were later extended to any representation by \name{Dufresne}, \name{Elmer}, and \name{Sezer} \cite{DuElSe2014Separating-invaria}.
\end{rem}
The proof of the theorem needs some preparation.

\subsection*{Invariants and covariants}\lab{inv-cov.subsec}
Let $V$ be as representation of $\SLtwo$.
The graded coordinate ring $\OOO(V)=\bigoplus_{d\geq 0}\OOO(V)_{d}$ is a locally finite and rational $\SLtwo$-module. A homogeneous irreducible submodule  $F \subset \OOO(V)_{d}$ is classically called a \emph{covariant of degree $d$ and weight $r$}, where $r$ is the weight of the highest weight vector $f_{0}$ of $F$. This means that $f_{0}$ is a homogeneous $\Ga$-invariant and that $t\cdot f_{0}= t^{r} f_{0}$ for $t \in \Gm$. In particular, $\dim F = r+1$. Thus, we always have $r\geq 0$, and $r=0$ if and only if $f_{0}$ is an $\SLtwo$-invariant. We will say that $f_{0}$ is a \emph{homogeneous $\Ga$-invariant of degree $d$ and weight $r$}. 

Clearly, the invariants $\OOO(V)^{\Ga}$ are linearly spanned by the homogeneous $\Ga$-invariants of degree $d$ and weight $r$ where $d,~r\geq 0$. Moreover, the homogeneous $\Ga$-invariants of degree $d$ and weight $r>0$ linearly span the plinth ideal  $\pp_{V}=\ker D \cap \im D$ where $D\in\VF(V)$ is the locally nilpotent vector field corresponding to the $\Ga$-action (see Definition~\reff{plinth.def}). This shows that the $\Ga$-invariants are generated by $\pp_{V}$ together with the $\SLtwo$-invariants.

In the following, we denote by $V[n]$ the irreducible $\SLtwo$-module of highest weight $n$, i.e., $\dim V[n] =n+1$. One can take $V[n]:=\k[x,y]_{n}$, the {\it binary forms of degree $n$}, with the standard linear action of $\SLtwo$.
It follows that the element $\sigma\in\SLtwo$ representing the non-trivial element of the Weyl group acts trivially on $V[n]_{0}$ if $n$ is odd or $n\equiv 0\pmod{4}$, and by $(-\id)$ if $n \equiv 2 \pmod{4}$.

In the proof below we will need the following classical result from invariant theory of binary forms. Choose a basis of weight vectors of $V[n]$ such that $\OOO(V[n]) = \CC[x_{0},x_{1},\ldots,x_{n}]$, where $x_{i}$ has weight $n-2i$. 

\begin{lem}[{see \cite[III.1.5]{Kr1984Geometrische-Metho}}]\lab{classical.lem}
As an $\SLtwo$-module we have the \name{Clebsch-Gordan} decomposition $\OOO(V[n])_{2}\simeq V[2n] \oplus V[2n-4] \oplus V[2n-8] \oplus \cdots$. The corresponding quadratic $\Ga$-invariants $f_{k}\in V[2n-4k]^{\Ga}$ have weight $2n-4k$ and are of the form 
$$
f_{k}=\alpha_{0}x_{0}x_{2k} + \alpha_{1}x_{1}x_{2k-1}+\cdots+\alpha_{k}x_{k}^{2}, \ k=0,1,2,\ldots,\lfloor n/2\rfloor,
$$ 
where all coefficients $\alpha_{j}$ are non-zero. 
\end{lem}
\begin{proof} For the binary forms $V[2k]$ of even degree $2k$ there is a unique quadratic $\SLtwo$-invariant which has the form
$A =  \gamma_{0}x_{0}x_{2k} + \gamma_{1}x_{1}x_{2k-1}+\cdots+\gamma_{k}x_{k}^{2}\in\CC[x_{0},\ldots,x_{k}]$ where all coefficients $\gamma_{i}$ are non-zero (see \cite[Satz 2.6]{Sc1968Vorlesungen-uber-I}; the invariant $A$ is classically called ``Apolare''). Now $\CC[x_{0},\ldots,x_{2k}]\subset \CC[x_{0},\ldots, x_{n}]=\OOO(V[n])$ is a $\Ga$-stable subalgebra, hence $A$ is a quadratic $\Ga$-invariant in $\OOO(V[n])$ of weight $2n-4k$, and so $f_{k}$ is a multiple of $A$.
\end{proof}

\begin{proof}[Proof of Theorem~\ref{NV.thm}]
(a) Denote by $\CC^{2}\simeq V[1]$ the standard representation of $\SLtwo$ and consider the closed embedding $V \hookrightarrow V \oplus \CC^{2}$ given by $v\mapsto (v,e_{1})$. Then we have the following diagram (see \cite[III.3.2]{Kr1984Geometrische-Metho}):
$$
\begin{CD}
V @>\phi>> W:=V\oplus \CC^{2}\\
@V{\pi}VV  @VV{\pi}V \\
V\quot\Ga @>\simeq>> W\quot\SLtwo
\end{CD}
$$
In particular, $\NNN_{V} = \phi^{-1}(\NNN_{W})=\NNN_{W}\cap V$. The \name{Hilbert}-Criterion tells us that the elements $w=(v,a)\in\NNN_{W}$ are characterized by the condition that $0\in \overline{\Gm\, gw}$ for a suitable $g\in\SL_{2}$ (see \cite[III.2.1]{Kr1984Geometrische-Metho}). This implies that $w=(v,e_{1})$ belongs to $\NNN_{W}$ if and only if $0\in\overline{\Gm v}$, i.e. if and only if $v\in V^{+}$.
\par\smallskip
(b) We first show that for every $v\in V \setminus (V^{+}\oplus V_{0})$ there is a homogeneous $\Ga$-invariant $f$ of weight $>0$ such that $f(v)\neq 0$. For that we can assume that $V$ is irreducible, i.e., $V = V[n]$. We have $\OOO(V) = \CC[x_{0},x_{1},\ldots,x_{n}]$, where $x_{i}$ has weight $n-2i$. Thus $x_{i}$ vanishes on $V^{+}$ if and only if $2i \leq n$, and $x_{i}$ vanishes on $V^{+}\oplus V_{0}$ if and only if $2i<n$.

Now let $v=(a_{0},a_{1},\ldots,a_{n}) \in V \setminus (V^{+}\oplus V_{0})$, and let $a_{k}$ be the first non-zero coefficient. Then the quadratic $\Ga$-invariant $f_{k}$ from Lemma~\reff{classical.lem} above gives $f_{k}(v) = \alpha_{k}a_{k}^{2}\neq 0$, and since $k<n/2$ the $\Ga$-invariant $f_{k}$ has a positive weight.

It remains to show that every homogeneous $\Ga$-invariant $f$ of weight $>0$ vanishes on $V^{+}\oplus V_{0}$. But this is clear, because every monomial $m=x_{0}^{d_{0}}x_{1}^{d_{1}}\cdots x_{n}^{d_{n}}$ of positive weight must contain an $x_{i}$ of positive weight, i.e., with $2i<n$. Hence $m$ vanishes on $V^{+}\oplus V_{0}$.

\par\smallskip
(c) The same argument shows that a homogeneous $\SLtwo$-invariant restricted to $V^{+}\oplus V^{0}$ does not depend on $V^{+}$. This implies that the induced morphism
$$
\pi|_{\PPP_{V}}\colon \PPP_{V}\to \pi(\PPP_{V})\subset V\quot\Ga
$$ 
is given by the $\SLtwo$-invariants and has the following factorization
$$
\begin{CD}
\PPP_{V}=V^{+}\oplus V_{0} @>{\pr}>> V_{0} @>{\pi_{\SLtwo}}>> \pi(\PPP_{V})=\pi(V_{0})\subset V\quot\SLtwo
\end{CD}
$$
where $\pi_{\SLtwo}\colon V \to V\quot \SLtwo$ is the quotient by $\SLtwo$. Now the claim follows from  the next lemma.
\end{proof}

\begin{lem}\lab{V0quotient.lem}
Let $V$ be a representations of $\SLtwo$ and $\pi_{\SLtwo}\colon V \to V\quot \SLtwo$ the quotient. Then $\pi_{\SLtwo}(V_{0}) \subset V\quot\SLtwo$ is closed and the induced morphism $V_{0}\to V\quot\SLtwo$ has a factorization
$$
\begin{CD}
V_{0} @>{\pi_{0}}>> V_{0}/W @>{\bar\pi}>> \pi_{\SLtwo}(V_{0})\subset V\quot\SLtwo
\end{CD}
$$
where $\pi_{0}$ is the quotient by $W$ and $\bar\pi$ is finite and bijective.
\end{lem}
\begin{proof} We first remark that the induced morphism $\pi':=\pi|_{V_{0}}\colon V_{0} \to \overline{\pi(V_{0})}$ is homogeneous and that ${\pi'}^{-1}(\pi(0)) = \{0\}$. Hence, $\pi(V_{0})\subset V\quot \SLtwo$ is closed and $\pi'$ is finite. It remains to see that the fibers of $\pi'$ are the $W$-orbits.

Since the orbits $\SLtwo v$ for $v\in V_{0}$ are closed, it suffices to show that we have $\SLtwo v \cap V_{0} = Wv$ for all $v\in V_{0}$. 
One easily reduces to the case where $V=V[2n]$, and then $V_{0}=\CC \, x^{n}y^{n}$. If $g (x^{n}y^{n}) \in \CC \, x^{n}y^{n}$ for some $g\in\SLtwo$, then either $gx \in \CC x$ and $gy\in \CC y$, or $gx\in \CC y$ and $gy \in \CC x$. In the first case, $g$ is diagonal and so $g (x^{n}y^{n})=x^{n}y^{n}$ and we are done. In the second case, $\sigma g$ is diagonal, and we are again done.
\end{proof}

\begin{rem} We were informed by \name{Gerald Schwarz} that Lemma~\reff{V0quotient.lem} holds for any representation $V$  of a reductive group $G$. If $\pi\colon V \to V\quot G$ is the quotient, then  the induced morphism $V_{0}\to V\quot G$ is finite and has a factorization
$$
\begin{CD}
\pi \colon V_{0} @>{\pi_{W}}>> V_{0}/W @>{\bar\pi}>> \pi(V_{0})\subset V\quot G
\end{CD}
$$
where $\pi_{W}$ is the quotient by the Weyl group $W$ and $\bar\pi$ is finite and bijective. 
\end{rem}

\bigskip
\section{The Separating Variety}\lab{sepvar.sec}
\subsection*{Definitions}  
In section~\reff{sep.sec}, we discussed  separating morphisms in the general context of a $G$-variety. We now
introduce the \emph{separating variety} $\SSS_{X}$ of a $G$-variety $X$, which measures how much the invariants separate the orbits. Set
$$
\SSS_{X}:=\{(x,y) \in X \times X \mid f(x) = f(y) \text{ for all } f\in \OOO(X)^{G}\} = \bigcup_{z\in X\quot G} \pi^{-1}(z)\times \pi^{-1}(z),
$$
where $\pi\colon X\to X\quot G$ is the quotient morphism. The separating variety first appeared in work of \name{Kemper} \cite[Section 2]{Ke2003Computing-invarian}. More schematically, the separating variety of $X$ is the reduced fiber product $(X \times_{X\quot G}X)_{\text{red}}$ (cf. \cite[Definition 2.2]{Du2009Separating-invaria}). If $Y \subset X$ is a $G$-stable subvariety, we write $\SSS_{X,Y}:=\SSS_{X}\cap (Y \times Y)$.

The separating variety $\SSS_{X}$ contains the closure of the graph
$$
\Gamma_{X}:=\{(gx,x) \mid g\in G, x\in X\} = \bigcup_{x\in X}Gx \times Gx \subset X \times X.
$$
Note that $\Gamma_{X}=\SSS_{X}$ exactly when the quotient $\pi$ is almost geometric, i.e., when all non-empty fibers of $\pi$ are orbits. Also, if $\Gamma_{X}$ is closed, then all orbits are closed and have the same dimension. (The first statement is clear, and the second follows since $Gx \times \{x\} = p_{2}^{-1}(x)$ where $p_{2}\colon \Gamma_{X}\to X$ is the second projection.)

More generally, we have the following result, which is a first step to determine the closure $\overline{\Gamma_{X}}$ and to decide whether $\overline{\Gamma_{X}}=\SSS_{X}$. For simplicity, we assume that $G$ is connected.  This implies among other things that $\overline{\Gamma_X}$ is irreducible. 

\begin{prop}\lab{sepvar.prop}
Let $G$ be connected and $X$ a normal affine $G$-variety. Assume that there is a dense open set $U \subset X\quot G$ such that $\phi^{-1}(u)$ is non-empty and contains a dense orbit for all closed points $u \in U$. Set $X':=\pi^{-1}(U)\subset X$ and $P:=X \setminus X'$.
\be
\item 
$\SSS_{X,P}$ is closed and $\SSS_{X} = \overline{\Gamma_{X}} \cup \SSS_{X,P}$. In particular, $\overline{\Gamma_{X}}$ is an irreducible component of $\SSS_{X}$.
\item If $\pi^{-1}(u)$ is a single orbit for every closed point $u \in U$, then
$$
\SSS_X=\Gamma_{X'} \cup \SSS_{X,P}=\Gamma_{X} \cup \SSS_{X,P}=\overline{\Gamma_{X}} \cup \SSS_{X,P}.
$$
\item \lab{sepvar.prop.smooth} 
Assume in addition that $X'$ is smooth, that the $G$-action on $X'$ is free, and that $\codim_{X} P > 1$. Then either $\Gamma_{X}$ is closed, or $\overline{\Gamma_{X}} \setminus \Gamma_{X'}$ has codimension 1 in $\overline{\Gamma_{X}}$.
\ee
\end{prop}

\begin{proof}
(a)  If $X\quot G$ is the disjoint union $O \cup A$, where $U$ is open and $A$ closed, then $\SSS_{X}=\SSS_{X,\pi^{-1}(O)}\cup\SSS_{X,\pi^{-1}(A)}$ where
$\SSS_{X,\pi^{-1}(O)}$ is open, $\SSS_{X,\pi^{-1}(A)}$ is closed, and the union is disjoint. Take $(x,y)\in \SSS_{X,X'}$. Then $\pi(x)=\pi(y)=:u\in U$. By assumption, the fiber $\pi^{-1}(u)$ contains a dense orbit, say $\overline{Gz}= \pi^{-1}(u)$. Hence, 
$$
(x,y)\in \pi^{-1}(u)\times \pi^{-1}(u) = \overline{Gz} \times \overline{Gz} = \overline{Gz\times Gz} \subseteq \overline{\Gamma_{X'}}=\overline{\Gamma_{X}}.
$$
It follows that $\SSS_{X}=\SSS_{X,X'}\cup\SSS_{X,P} = \overline{\Gamma_{X}} \cup \SSS_{X,P}$.

\par\smallskip
(b) Since the fibers over $U$ are orbits, we get $\SSS_{X,X'}=\Gamma_{X'}= \Gamma_{X} \cap (X'\times X')$, and so 
$$
\SSS_{X}=\SSS_{X,X'}\cup \SSS_{X,P} = \Gamma_{X'}\cup \SSS_{X,P}.
$$
The claim follows.

\par\smallskip
(c) Consider the morphism $\mu\colon G \times X \to X\times X$, $(g,x)\mapsto(gx,x)$, whose image is $\Gamma_{X}$. By assumption, it induces an isomorphism $\mu_{0}\colon G \times X' \simto \Gamma_{X'}$, and thus, a birational morphism $\tilde\mu\colon G\times X \to \tilde \Gamma$, where $\tilde\Gamma\to\overline{\Gamma_{X}}$ is the normalization. If $\codim_{\overline{\Gamma_{X}}}\overline{\Gamma_{X}}\setminus\Gamma_{X'} >1$, then by \name{Igusa}'s criterion \cite{Ig1973Geometry-of-absolu} (cf. \cite[Appendix~A, Proposition~5.12]{Kr2011Algebraic-Transfor}), $\tilde\mu$ is an  isomorphism,  and so $\Gamma_{X}$ is closed.
\end{proof}

\begin{rem} 
The first statement of the proposition above has the following converse: \emph{ If $\overline{\Gamma_{X}}$ is an irreducible component of $\SSS_{X}$, then the general fiber of $\pi\colon X \to X\quot G$ contains a dense orbit.} 

In order to see this, we can replace $X \quot G$ be a dense open set and thus assume that $X\quot G$ is affine algebraic, $\pi\colon X\to X\quot G$ is flat, and the fibers are irreducible of dimension $n$. Then every irreducible component of $\SSS_{X}= X \times_{X\quot G} X$ has dimension $2\dim X - \dim X\quot G=\dim X + n$ (see \cite[Cor. 9.6 in Chap.~III]{Ha1977Algebraic-geometry}). On the other hand, $\dim\overline{\Gamma_{X}} = \dim X + d$ where $d:=\max\{\dim Gx\mid x\in X\}$. Hence $n=d$ and so the general fiber contains a dense orbit.
\end{rem}

\subsection*{The case of $\Ga$-varieties}
If $X$ is a $\Ga$-variety, then by Proposition~\reff{plinth.prop},  the quotient $\pi\colon X\setminus \PPP_{X} \to \pi(X\setminus \PPP_{X})$ is a $\Ga$-bundle. This implies the following corollary.

\begin{cor}\lab{SX.lem}
If $X$ is a normal affine $\Ga$-variety, then
\[
\SSS_X=\Gamma_{X\setminus \PPP_X} \cup \SSS_{X,\PPP_X} =\Gamma_X \cup \SSS_{X,\PPP_X}=\overline{\Gamma_X} \cup \SSS_{X,\PPP_X},
\]
and $\overline{\Gamma_X}$ is an irreducible component of $\SSS_{X}$.
\end{cor}

In the remaining part of this section, we determine the irreducible components of $\SSS_{V}$ for a representation  $V$ of $\Ga$ (cf. \cite{DuKo2013The-separating-var}, where this is done for indecomposable representations). We have seen in Theorem~\reff{NV.thm}(c) that the
image $\pi(\PPP_{V}) \subset V\quot\Ga$ is closed and the induced morphism
$\pi|_{\PPP_{V}}\colon \PPP_{V}\to \pi(\PPP_{V})$ has a factorization
\[\tag{$*$}
\begin{CD}
\PPP_{V}=V^{+}\oplus V_{0} @>{\pr}>> V_{0} @>{\pi_{0}}>> V_{0}/W @>{\bar\pi}>> \pi(\PPP_{V}),
\end{CD}
\]
where $\pi_{0}$ is the quotient by $W$ and $\bar\pi$ is finite and bijective.
If $v \in \PPP_{V}=V_{0}\oplus V^{+}$, we denote by $v_{0}$ the component of $v$ in $V_{0}$. Define the following closed subsets of $\SSS_{\PPP_{V}}$:
$$
C:= \{(v,v')\in \PPP_{V}\times\PPP_{V}\mid v'_{0}=v_{0}\}, 
\quad C_{\sigma}:= \{(v,v')\in\PPP_{V}\times\PPP_{V}\mid v'_{0}=\sigma(v_{0})\}.
$$
Both are irreducible and isomorphic to $V_{0}\times(V^{+}\times V^{+})$. Now the factorization $(*)$ implies the following result.
\begin{lem}\lab{sepvar.lem}
\be
\item If $\sigma$ acts trivially on $V_{0}$, then $\SSS_{\PPP_{V}}=C = C_{\sigma}$ is irreducible.
\item If $\sigma$ acts non-trivially on $V_{0}$, then $\SSS_{\PPP_{V}}=C \cup C_{\sigma}$ has  two irreducible components.
\ee
In particular, $\SSS_{\PPP_{V}}$ is equidimensional of dimension $\dim V$.
\end{lem}

Now we can formulate our main result about the separating variety $\SSS_{V}$.
\begin{thm}\lab{sepvarGa.thm} 
We have  $\SSS_{V}=\overline{\Gamma_{V}}$ if and only if the Weyl group acts trivially on $V_{0}$, or if $V = V[2]\oplus \CC^{m}$. Otherwise, $\SSS_{V}$ has two irreducible components:
$$
\SSS_{V}=\overline{\Gamma_{V}} \cup C,
$$
where $\dim \overline{\Gamma_{V}} = \dim V+1$ and $\dim C = \dim V$.
\end{thm}
\begin{proof} We can assume that $V^{\SLtwo}=(0)$. In fact, if $V = W \oplus \CC^{m}$, then $\Gamma_{V} = \Gamma_{W}\times \CC^{m}$ and $\SSS_{V}=\SSS_{W}\times \CC^{m}$. It is easy to see that for $V=V[2]$ we have $\SSS_{V}=\overline{\Gamma_{V}}$. In all other cases, we have $\dim V^{+}\geq 2$ which implies that the component $C$ is not contained in $\Gamma_{V}$. On the other hand, $\overline{\Gamma_{V}}=\Gamma_{V}\cup C_{\sigma}$ by Lemma~\reff{basic.lem} below, and the claim follows from Lemma~\reff{sepvar.lem}.
\end{proof}

\begin{lem}\lab{basic.lem}
We have $\overline{\Gamma_{V}}=\Gamma_{V}\cup C_{\sigma}$. 
\end{lem}
The proof needs some preparation. If $X$ is a variety and $R$ a $\CC$-algebra, we define the $R$-valued points by $X(R):=\Mor(\Spec R, X)$. We have a canonical inclusion $X(\Kt)\subset X(\Kf)$ and a canonical map $X(\Kt) \to X(\CC)=X$ which will be denoted by $x=x(t) \mapsto x(0)=x|_{t=0}$. We will constantly use the following fact. If $\phi\colon X \to Y$ is a morphism and $y\in\overline{\phi(X)}$, then there is an $x=x(t)\in X(\Kf)$ such that $\phi(x)\in Y(\Kt)$ and $\phi(x)|_{t=0}=y$. Moreover, if $y\notin \phi(X)$, then $x\notin X(\Kt)$.

\begin{proof}[Proof of Lemma~\ref{basic.lem}]
We know from Proposition~\reff{sepvar.prop} that $E:=\overline{\Gamma_{V}} \cap \SSS_{\PPP_{V}}=\overline{\Gamma_{V}}\setminus\Gamma_{\Vbd}$ has codimension 1 in 
$\overline{\Gamma_{V}}$, hence $E$ is either $C$, $C_{\sigma}$, or $C\cup C_{\sigma}$ by Lemma~\reff{sepvar.lem}.

We now show that $\overline{\Gamma_{V}}\setminus\Gamma_{V}\subset C_{\sigma}$, which implies that $\overline{\Gamma_{V}}=\Gamma_{V}\cup C_{\sigma}$, hence the claim. Let $(v',v'')\in \overline{\Gamma_{V}}\setminus\Gamma_{V}\subset \SSS_{\PPP_{V}}$. 
Since $\Gamma_{V}$ is the image of the morphism
$\mu\colon \Ga \times V \to V \times V$, $(s,v)\mapsto (sv,v)$, there are element $s(t)\in\Ga(\Kf)\setminus\Ga(\Kt)$ and $v(t)\in V(\Kt)$ such that the following holds:
\be
\item $v(0) = v''$;
\item $s(t)v(t)\in V(\Kt)$ and $(s(t)v(t))|_{t=0}= v'$.
\ee
If $v''_{0}\in (V_{0})^{\sigma}$, then $v'_{0}=v''_{0}=\sigma v''_{0}$, and so $(v',v'') \in C_{\sigma}$. Thus we can assume that $v''_{0}$ is not fixed by $\sigma$, and we have to show that  $v'_{0}=\sigma v''_{0}=-v''_{0}$.
\par\smallskip
Now we use \name{Luna}'s  Slice Theorem in the point $v''_{0}$.  Denote by $T \subset \SLtwo$ the diagonal matrices identified with $\Gm$ as above, and by $U \subset \SLtwo$ the upper triangular unipotent matrices, which we can identify with $\Ga$. 
There is $T$-stable subspace $W \subset V$ containing $v''_{0}$ such that the morphism $\mu\colon \SLtwo *_{T} W \to V$ given by $\mu([g,w]):=gw$ is \'etale in a $\SLtwo$-saturated open neighborhood of $[e,v''_{0}]$ (see \cite{Sl1989Der-Scheibensatz-f}). 
Here the bundle $\SLtwo *_{T} W$ is the quotient $(\SLtwo \times W)\quot T$ under the action $t(g,w) := (gt^{-1},tw)$, and the quotient morphism $\SLtwo \times W  \to \SLtwo *_{T} W$ is a principal $T$-bundle. This implies that we can lift the elements $v(t)$ and $s(t)v(t)$ to $\SLtwo \times W$, i.e., 
there are elements $g(t)\in\SLtwo(\Kt)$, $w(t)\in W(\Kt)$ and $p(t)\in T(\Kf)$ such that the following holds:
\be
\item[(a$'$)] $g(t)w(t) = v(t)$, hence $g(0) w(0)=v''$;
\item[(b$'$)] $\tilde g(t):=s(t)g(t)p(t)^{-1}\in\SLtwo(\Kt)$ and  $\tilde w(t):=p(t) w(t) \in W(\Kt)$, hence $\tilde g(t)\tilde w(t) = s(t) v(t)$ and $\tilde g(0) \tilde w(0) = v'$.
\ee
Setting 
$$
s(t) = \begin{bmatrix} 1 & f(t)\\ 0&1 \end{bmatrix},\quad
g(t) = \begin{bmatrix} a(t) & b(t)\\ c(t)& d(t) \end{bmatrix},\quad
p(t)=\begin{bmatrix} r(t) & 0\\ 0& r(t)^{-1} \end{bmatrix},
$$
where $f(t) \in \Kf\setminus\Kt$, $a(t),b(t),c(t),d(t) \in \Kt$, and $r(t)\in\Kf$, we get 
$$
\tilde g(t)=s(t)g(t)p(t)^{-1} = \begin{bmatrix} r^{-1}(a+fc) & (b+df)r \\ r^{-1}c & dr \end{bmatrix}.
$$
Obviously, $p(t)\notin T(\Kt)$, since $s(t)\notin U(\Kt)$. Thus either $r(t)\in t\Kt$ and $c(0)=0$, or $r(t)^{-1}\in t\Kt$ and $d(0)=0$. In the first case we get 
\[ \tag{1}
g(0)\in \left\{\begin{bmatrix} * & * \\ 0 & * \end{bmatrix}\in\SLtwo\right\}=:B \text{ \ and \ } 
\tilde g(0) \in \left\{\begin{bmatrix} * & * \\ * & 0 \end{bmatrix}\in\SLtwo\right\} = B\sigma,
\]
and in the second
\[ \tag{2}
g(0)\in B\sigma \text{ \ and \ } 
\tilde g(0) \in B.
\]
Moreover, since $\tilde w(t) = p(t) w(t)$, we get $\tilde w(0)_{0}=w(0)_{0}$. 
Also note that for any $b\in B$ and $u\in V_{0}\oplus V^{+}$ we have $(bu)_{0}=u_{0}$. 

Assume now that we are in case (1). Since $g(0) w(0) = v'' \in V_{0}\oplus V^{+}$, we get $w(0) \in V_{0}\oplus V^{+}$, hence $w(0)_{0} = (g(0)w(0))_{0} = v''_{0}$. On the other hand, $\tilde g(0) \in B\sigma$ and $\tilde g(0) \tilde w(0) = v' \in V_{0}\oplus V^{+}$, hence $\sigma \tilde w(0) \in V_{0}\oplus V^{+}$ and $(\sigma \tilde w(0))_{0}= v'_{0}$. Thus $v_{0}' = \sigma \, \tilde w(0)_{0}= -\tilde w(0)_{0}=-w(0)_{0}=-v''_{0}$, i.e. $(v',v'')\in C_{\sigma}$,  and the claim follows. Case (2) is similar.
\end{proof}

\par\bigskip
\section{$\Ga$-actions on $\SLtwo$-varieties}\lab{Ga-SL2.sec}

In this section, we generalize some of the results obtained for representations of $\Ga$ to affine $\SLtwo$-varieties. As in 
section~\reff{rep.sec} we identify  $\Ga$ with the unipotent subgroup $U\subset\SLtwo$ via $s \mapsto \begin{bmatrix} 1 & s \\ 0 & 1 \end{bmatrix}$, and $\Gm$ with the maximal torus $T \subset \SLtwo$ via $t\mapsto \begin{bmatrix} t & 0 \\ 0 & t^{-1} \end{bmatrix}$. 
Thus every  $\SLtwo$-variety $X$ can be regarded as a $\Ga$-variety. These $\Ga$-varieties have some very special properties, e.g. the following classical result which was already used in the proof of Theorem~\reff{NV.thm} (see \cite[III.3.2]{Kr1984Geometrische-Metho}).

\begin{lem}\lab{Ga-inv.lem}
Let $X$ be an affine $\SLtwo$-variety and denote by $\k^{2}$ the standard representation of $\SLtwo$. Then the closed $\Ga$-equivariant embedding $X \into X \times \k^{2}$, $x\mapsto (x,e_{1})$, induces an isomorphism $X\quot \Ga \simto (X\times \k^{2})\quot \SLtwo$. In particular, the $\Ga$-invariants $\OOO(X)^{\Ga}$ are finitely generated.
\end{lem}
An immediate consequence is that for every  closed embedding $X\into Y$ of affine $\SLtwo$-varieties the induced map $X\quot\Ga\to Y\quot \Ga$ is also a closed embedding. 

\begin{prop} Let $V$ be a representation of $\SLtwo$ and $X \subset V$ a closed $\SLtwo$-stable subset.
\begin{enumerate}
\item $\SSS_X=\SSS_V\cap (X\times X)$.
\item For any  $v\in (V_0\oplus V^+) \cap X$ we have  $v_0 \in X$.
\item $\PPP_{X}=\PPP_{V}\cap X$.
More precisely, the image of the plinth ideal $\pp_{V}$ under the restriction map is the plinth ideal $\pp_{X}$.
\item $\SSS_{X,\PPP_X}=\SSS_{V,\PPP_V}\cap (X\times X)$.
\end{enumerate}
\end{prop}
\begin{proof}
(a) The inclusion $\SSS_X\subseteq\SSS_V\cap (X\times X)$ is obvious. Take $(x,x')\in \SSS_V\cap (X\times X)$. We have $f(x)=f(x')$ for all $f\in \OOO(V)^{\Ga}$. Since every element in $\OOO(X)^{\Ga}$ is the restriction to $X$ of an element in $\OOO(V)^{\Ga}$, we get $h(x)=h(x')$ for all $h\in\OOO(X)^{\Ga}$, and so $(x,x')\in \SSS_X$.

(b) Note that  $v_0\in \overline{\Gm  v}$, and the claim follows, since $X$ is closed and $\SLtwo$-stable.

(c) The restriction map $\OOO(V) \to \OOO(X)$ is $\SLtwo$-equivariant and so the image of an irreducible $\SLtwo$-subrepresentation $W \subset \OOO(V)$ is either $(0)$ or isomorphic to $W$. Therefore, the generators of $\pp_{V}$ are mapped onto the generators of $\pp_{X}$.

(d) This is clear from what has been said so far.
\end{proof}

\begin{prop}
Let $V$ be a representation of $\SLtwo$ and $X \subset V$ a closed $\SLtwo$-stable subset. Set $X_{0}:= X^{\Gm}=X \cap V_{0}$. Then the following are equivalent:
\be
\item[(i)] $\SSS_X=\graphXbar$;
\item[(ii)] $\graphXbar=\graphVbar \cap (X\times X)$ and $(x_0+V^+) \cap X=\Ga x_0$ for all $x_0\in X_{0}\setminus (X_{0})^{\sigma}$.
\ee
\end{prop}
\begin{proof}
Since $\graphX = \graphV \cap X$ and  $\graphVbar = \graphV \cup C_{\sigma}$ (Lemma~\reff{basic.lem}), we get 
$$
\graphXbar \subseteq \graphVbar \cap (X\times X) = \graphX \cup (C_{\sigma}\cap (X \times X)) \subseteq \SSS_{X},
$$
and from $\SSS_{V}=\graphV\cup C_{\sigma} \cup C$ we obtain
$$
\SSS_{X}= \SSS_{V}\cap (X \times X) = \graphX \cup (C_{\sigma}\cap (X \times X))\cup (C \cap (X \times X)).
$$
Therefore, $\graphXbar = \SSS_{X}$ if and only if $\graphXbar \supseteq \graphVbar \cap (X\times X)$ and $(C\setminus C_{\sigma})\cap (X \times X) \subset \graphX$. But the latter condition is clearly equivalent to $(x_0+V^+) \cap X=\Ga x_0$ for all $x_0\in X_{0}\setminus (X_{0})^{\sigma}$.
\end{proof}

\begin{exa}\lab{SLtwoT.exa}
Let $X := \SLtwo/T$ where $T\subset \SLtwo$ is the group of diagonal matrices acting by right multiplication on $\SLtwo$. This variety is the so-called \name{Danielewski} surface, i.e.,  the smooth 2-dimensional affine quadric $X = \VVV(xz-y^{2}+y)\subset \k^{3}$ (cf. \cite{DuPo2009On-a-class-of-Dani}), and the quotient map is given by
$$
\pi_{\SLtwo}\colon \SLtwo \to X, \quad  \begin{bmatrix} a&b\\c&d\end{bmatrix}\mapsto (ab,ad,cd).
$$
Clearly, $X$ is an $\SLtwo$-variety where the action is induced by left multiplication on $\SLtwo$, and thus a $\Ga$-variety. The quotient by $\Ga$ is $\Aone$, and the quotient map is given by 
$$
\SLtwo/T \ni \begin{bmatrix} a&b\\c&d\end{bmatrix} T \mapsto cd, \ \text{ i.e. } X \ni (x,y,z) \mapsto z.
$$
The plinth ideal is generated by $z$ and is reduced. The plinth variety $\PPP_{X}$ consists of the two orbits $O_{1}:=UT$ and $O_{2}:=U\sigma T$ where $\sigma := \left[\begin{smallmatrix} 0&-1\\1&0\end{smallmatrix}\right]$, and so $X_{\alg}=X \setminus (O_{1}\cup O_{2})$. Moreover, the induced morphisms $X\setminus O_{i}\to \Aone$ are both trivial $\Ga$-bundles, and so $X\setminus O_{i}\simeq \AA^{2}$ for $i=1,2$. Thus $\Xbd = X$, but $\pi\colon X \to \Aone$ is not a $\Ga$-bundle, because $\pi^{-1}(0) = O_{1}\cup O_{2}$. It follows that
$$
\Gamma_{X}=\bigcup_{O \text{ orbit}} O\times O \text{ \ is open in } \SSS_{X}=\Gamma_{X}\cup (O_{1}\times O_{2})\cup (O_{2}\times O_{1}).
$$
Since $\SSS_{X} \subset X \times X$ is the hypersurface defined by $f:=\pi\circ\pr_{1}-\pi\circ\pr_{2}$. Finally, it follows from \name{Krull}'s Principal Ideal Theorem (see \cite[Appendix A, Theorem 3.13]{Kr2011Algebraic-Transfor} or \cite[Chapter II, Theorem 10.1]{Ei1995Commutative-algebr}) that $\SSS_{X}=\overline{\Gamma_{X}}$ is irreducible.
\end{exa}

\begin{exa}\lab{SLtwoN.exa}
Now let us look at $Y :=\SLtwo/N$, where $N = T \cup \sigma T$ is the normalizer of $T$. Then $\sigma$ induces an automorphism of order 2 on $X = \SLtwo/T$ commuting with the $\Ga$-action, and the automorphism $-\id$ on the quotient $X\quot \Ga = \Aone$. Thus $Y = X / \langle\sigma\rangle$ and $Y\quot\Ga = \Aone/\{\pm\id\} \simeq \Aone$. Since $\sigma(O_{1}) = O_{2}$ in the notation of Example~\reff{SLtwoT.exa} we see that the plinth variety $\PPP_{Y}=\pi^{-1}(0)$ is a single orbit, but the plinth ideal $\pp_{Y}$ is not prime. Therefore, $\pi\colon Y \to \Aone$ is a geometric quotient, but not a principal $\Ga$-bundle. In this case, $Y_{\bd} = Y_{\alg} = X\setminus\PPP_{Y}$, and $\SSS_{Y}=\Gamma_{Y}$. 
\end{exa}

\par\bigskip
\section{\name{Roberts}' example}\lab{Roberts.sec}
In this section we discuss \name{Roberts}' counterexample to \name{Hilbert}'s fourteenth problem (\cite{Ro1990An-infinitely-gene}, cf. \cite{AC1994Note-on-a-countere}). We assume that $\Char \k = 0$ and 
define an action of the additive group $\Ga$ on $\AA^7$ as follows:
\[
s\cdot (a_1,a_2,a_3,b_1,b_2,b_3,c):=(a_1,a_2,a_3,b_1+sa_1^3,b_2+sa_2^3,b_3+sa_3^3,c+s(a_1a_2a_3)^2).
\]
It corresponds to the locally nilpotent vector field 
\[
D:=x_1^{3}\frac{\partial}{\partial y_1}+ x_2^{3}\frac{\partial}{\partial y_2}+ x_3^{3}\frac{\partial}{\partial y_3}+ (x_1x_2x_3)^{2}\frac{\partial}{\partial z},
\]
where we use the coordinates $\oA=\kk[x_1,x_2,x_3,y_1,y_2,y_3,z]$. Put $A:=\oA^{\Ga}$ and denote the quotient morphism by 
$\pi\colon {\AA^7} \to {\AA^7\quot \Ga}:=\Spec(A)$. 

The $x_{i}$ are invariants, and $D(y_{i})=x_{i}^{3}$,  hence $x_{i}^{3}\in \pp_{\AA^7}$ and $(x_1,x_2,x_3)\subset \ff_{\AA^{7}\quot\Ga}$. It follows that  $\AA^7_{x_i} \to (\AA^{7}\quot \Ga)_{x_{i}}$ is a trivial $\Ga$-bundle for $i=1,2,3$. This allows to find the following additional invariants:
\begin{gather*}\tag{$*$}\label{eqn-A}
u_{12}  := x_1^{3}y_2-x_2^{3}y_1,\ 
u_{13} := x_1^{3}y_3-x_3^{3}y_1, \
u_{23}  := x_2^{3}y_3-x_3^{3}y_2,\\
\bbb{1,1} := x_1z-x_2^2x_3^2y_1,\ 
\bbb{2,1} := x_2z-x_1^2x_3^2y_2,\ 
\bbb{3,1} := x_3z-x_1^2x_2^2y_3.
\end{gather*}
Define the following subalgebras of the ring of invariants $A$: 
$$
A_{0}:=\kk[x_{1},x_{2},x_{3},u_{12},u_{13},u_{23}]\subset A_{1}:=A_{0}[\bbb{1,1},\bbb{2,1},\bbb{3,1}] \subset A.
$$
We then have $(A_1)_{x_i}=\oAga_{x_i}=\OOO(\AA^{7}_{x_{i}})^{\Ga}$. Using a symbolic computation software like \name{Singular} \cite{DeGrPf2012sc-Singular-3-1-6-}, it is easy to see that $Y_{0}:=\Spec A_{0}\subset \AA^{6}$ is the normal hypersurface defined by the equation $x_{1}^{3}u_{12}+ x_{2}^{3}u_{13}+x_{3}^{3}u_{23}=0$, and that $Y_{1}:= \Spec A_{1} \subset \AA^{9}$ has dimension 6 and its ideal $I(Y_{1})$ is generated by the following $5$ functions:
\begin{gather*}
x_{1}^{2}u_{12}-x_{3}\bbb{2,1}+x_{2}\bbb{3,1}, \ x_{2}^{2}u_{13}-x_{1}\bbb{3,1}+x_{3}\bbb{1,1}, \ x_{3}^{2}u_{23}-x_{2}\bbb{1,1}+x_{1}\bbb{2,1}, \\
u_{12}u_{13}u_{23}(x_{1}x_{2}\bbb{3,1}+ x_{2}x_{3}\bbb{1,1}+x_{3}x_{1}\bbb{2,1}) + u_{12}\bbb{1,1}^{3}+u_{13}\bbb{2,1}^{3}+u_{23}\bbb{3,1}^{3},\\
x_{1}x_{2}x_{3}u_{12}u_{13}u_{23} + x_{1}u_{12}\bbb{1,1}^2 + x_{2}u_{13}\bbb{2,1}^2 + x_{3}u_{23}\bbb{3,1}^2.
\end{gather*}
\begin{lem}\lab{normal.lem}
The variety $Y_{1}$ is normal.
\end{lem}
\begin{proof} Again, using for example \name{Singular} \cite{DeGrPf2012sc-Singular-3-1-6-}, one verifies that the ideal $x_{1}A_{1}$ is radical. Let $f \in Q(A_{1})$ be integral over $A_{1}$, that is, suppose $f$ satisfies an equation
\begin{equation*}\label{integral}
f^{d}=a_{1}f^{d-1}+a_{2}f^{d-2}+\cdots+a_{d},
\end{equation*}
where $a_{i}\in A_{1}$. Since $(A_{1})_{x_{1}}$ is normal, we have $x_{1}^{m}f \in A_{1}$ for some $m\geq 0$. We choose a minimal $m$ with this property. It follows from the equation above that $(x_{1}^{m}f)^{d}\in x_{1}A_{1}$, hence $x_{1}^{m}f\in x_{1}A_{1}$, and thus $f\in A_{1}$, because of the minimality of $m$.
\end{proof}
The action of $\Ga$ on ${\AA^7}$ commutes with the $(\Gm)^3$-action with weights
$$
(1,0,0), (0,1,0), (0,0,1), (3,0,0), (0,3,0), (0,0,3), (2,2,2),
$$ 
and so  $(\Gm)^3$ also acts on ${\AA^7\quot \Ga}$. As the polynomials in (\reff{eqn-A}) are multi-homogeneous, $(\Gm)^3$ also acts on $Y_{0}$ and $Y_{1}$. 

The following propositions collects the main properties of  $\pi\colon {\AA^7} \to {\AA^7}\quot \Ga$. Most statements follow immediately from what we have done so far. The difficult part is the description of the finite generation ideal $\fxga$. Recall that $\PA \subset \AA^{7}$ denotes the plinth variety  (see Definition~\reff{plinth.def}) and $\SSS_{\AA^{7}} \subset \AA^{7}\times\AA^{7}$ the separating variety (see section~\reff{sepvar.sec}).
\begin{prop}\lab{Roberts1.prop}
\be
\item\lab{Xbd} $\PA=(\AA^{7})^{\Ga}=\VVV_{\AA^7}(x_{1},x_{2},x_{3})\simeq\AA^{4}$, and 
$$
(\AA^7)_{\bd}={\AA^7}\setminus \PA= (\AA^7)_{x_1}\cup (\AA^7)_{x_2} \cup (\AA^7)_{x_3}.
$$
\item\lab{Impi} $\pi((\AA^{7})^{\Ga}) = \{\pi(0)\}$, and 
$$
\pi({\AA^7})=({\AA^7\quot \Ga})_{x_1} \cup ({\AA^7\quot \Ga})_{x_2} \cup ({\AA^7\quot \Ga})_{x_3} \cup \{\pi(0)\}
= \pi((\AA^{7})_{\bd}) \cup \{\pi(0)\}.
$$
\item\lab{R7sepvar} The separating variety $\SSS_{\AA^7}$ has two irreducible components: 
$$
\SSS_{\AA^{7}}=\overline{\Gamma_{\AA^7}} \cup (\PA\times\PA),
$$
both of  dimension 8.
\item\lab{xgaalg} We have $\fxga=\sqrt{(x_1,x_2,x_3)}$, and so 
$$
({\AA^7\quot \Ga})_{\alg}=({\AA^7\quot \Ga})_{x_1}\cup ({\AA^7\quot \Ga})_{x_2} \cup ({\AA^7\quot \Ga})_{x_3} = \pi((\AA^{7})_{\bd}).
$$ 
In particular, $({\AA^7\quot \Ga})_{\alg}$ is algebraic.  
\item\lab{zeroset} $\AA^{7}\quot\Ga \setminus (\AA^{7}\quot\Ga)_{\alg}=\VVV_{{\AA^7}\quot\Ga}(x_{1},x_{2},x_{3}) \simeq \AA^{3}$ has codimension 3 in $\AA^7\quot \Ga$. 
\item\lab{R7nice}  $\AA^7\quot \Ga$ is Jacobson, and its closed points coincide with its rational points.   
\ee
\end{prop}
The inclusion $A_{1}\subset A$ defines an invariant morphism $\phi\colon {\AA^7} \to Y_{1}$ which factors through the quotient $\pi$:
\begin{diagram}
{\AA^7} & \rTo^{\pi}& {\AA^7}\quot\Ga\\
& \rdTo_{\phi} & \dTo_{\bar\phi} \\
&& Y_{1}
\end{diagram}
\begin{prop}\lab{Roberts2.prop}
\be
\item $\overline{\phi}$ induces an isomorphism $\pi((\AA^7)_\bd)\simto \phi((\AA^7)_\bd)$. \lab{iso}
\item $Y_{1}$ is normal and $\overline{\phi}\colon {\AA^7}\quot \Ga \to Y_{1}$ is injective on $\pi({\AA^7})$. In particular, $\phi$ is a separating morphism. \lab{sep}
\item  $\oAga=\OOO_{Y_{1}}(Y_{1}\setminus \VVV_{Y_{1}}(x_1,x_2,x_3))$.\lab{quasiaffine}
\ee
\end{prop}
A proof that $\phi$ is a separating morphism and that (\reff{quasiaffine}) holds already appeared in \cite[Example 4.2]{Du2013Finite-separating-}.

\begin{proof}[Proof of Proposition~\ref{Roberts1.prop}(\ref{Xbd})--(\ref{R7sepvar})]
We have $\pi^{-1}(\VVV_{{\AA^7\quot \Ga}}(x_1,x_2,x_3))=\pi^{-1}(\pi(0))$, implying (\reff{Xbd}). As $\overline{\pi(\pi^{-1}(\VVV_{{\AA^7\quot \Ga}}(x_1,x_2,x_3)))}=\{\pi(0)\}$, (\reff{Impi}) follows.  Finally, statement (\reff{R7sepvar}) follows from (\reff{Xbd}) and Corollary \reff{SX.lem}.
\end{proof}
The proofs the remaining  statements (\reff{xgaalg})--(\reff{R7nice}) need some preparation. They will be given at the end of the section.

\begin{proof}[Proof of Proposition~\ref{Roberts2.prop}]
Statement (\reff{iso}) holds since $({\AA^7\quot \Ga})_{x_i}\cong (Y_{1})_{x_i}$.
We have seen in Lemma~\reff{normal.lem} that $Y_{1}$ is normal and the morphism $\bar\phi$ is injective on $\pi({\AA^7})$, since $\bar\phi(\pi(x_0))=\phi(x_0)\in \VVV_{Y_{1}}(x_1,x_2,x_3)$, proving (\reff{sep}). Finally, (\reff{quasiaffine}) follows from Theorem~\reff{genthm}(d) since $Y_{1}\setminus\phi({\AA^7}) = \VVV_{Y_{1}}(x_{1},x_{2},x_{3})\setminus\{\phi(0)\}$.  
\end{proof}

To prove that $\oAga$ is not finitely generated, \name{Roberts} showed in \cite[Lemma 3]{Ro1990An-infinitely-gene} that there exist invariants of the form 
$$
x_iz^n+\textrm{terms of lower } z\textrm{-degree}
$$ 
for $i=1,2,3$ and $n\geq 0$.  
Later, \name{Kuroda} proved (see  \cite[Theorem 3.3]{Ku2004A-generalization-o}) that any set $S$ of such invariants, together with $u_{12}, u_{13}, u_{23}$,  forms a SAGBI-basis for the lexicographic monomial ordering with $x_1\prec x_2\prec x_3 \prec y_1 \prec y_2 \prec y_3 \prec z$. We will improve this statement in  Lemma~\reff{SAGBI.lem} below.

Recall that if $R$ is a subalgebra of a polynomial ring, then for a given monomial ordering, a {\it SAGBI-basis}  is a subset $S\subset R$ such that $\kk[\LT(S)]=\kk[\LT(R)]$ where $\LT(S)$ denotes the set of leading terms of the polynomials in $S$ (see \cite{RoSw1990Subalgebra-bases}). Such a basis always generates $R$. Note that for \name{Kuroda}'s SAGBI-bases $S$ defined above we always have 
$$
\LT(S)=\{x_{1}^{3}y_{2},x_{1}^{3}y_{3},x_{2}^{3}y_{3}, x_{j}z^{n}\mid j\in\{1,2,3\}, n\geq 0\}.
$$

\begin{lem}\lab{invar.lem}
There exist invariants $\bbb{i,n}$ for $n\geq 0$ and $i=1,2,3$ which are multi-homogeneous and of the form
\begin{multline*}\tag{$*$}
\bbb{i,n}=x_{i}z^{n}-nx_{j}^{2}x_{k}^{2}y_{i}z^{n-1} +
\\+\binom{n}{2}(x_{i}^{2}x_{j}^{4}x_{k}y_{i}y_{k}+x_{i}^{2}x_{j}x_{k}^{4}y_{i}y_{j}-x_{i}^{5}x_{j}x_{k}y_{j}y_{k})z^{n-2} +\\
+ \text{ terms of lower $z$-degree}
\end{multline*}
\end{lem}

\begin{proof} By symmetry it suffices to look at the case $i=1$.
We know from \cite[Lemma 3]{Ro1990An-infinitely-gene} that invariants $\bbb{1,n}$ with leading term $x_{1}z^{n}$ exist, and we can clearly assume that they are multi-homogeneous of degree $(2n+1,2n,2n)$, hence
$$
\bbb{1,n} = x_{1}z^{n}+f_{1}z^{n-1}+f_{2}z^{n-2} + \text{\it \ terms of lower $z$-degree}
$$
where $f_{1},f_{2}\in \kk[x,y]:=\kk[x_{1},x_{2},x_{3},y_{1},y_{2},y_{3}]$, $\deg f_{1}=(3,2,2)$ and $\deg f_{2}=(5,4,4)$. From $D(\bbb{1,n})=0$ we get the following differential equations
\begin{eqnarray*}
D(f_{1}) &=& -nx_{1}D(z) = -nx_{1}^{3}x_{2}^{2}x_{3}^{2} \\
D(f_{2}) &=& -(n-1)f_{1}D(z) = -(n-1)x_{1}^{2}x_{2}^{2}x_{3}^{2}f_{1}, 
\end{eqnarray*}
which have the special solutions 
$$
h_{1}:=-nx_{2}^{2}x_{2}^{2}y_{1}, \textrm{ and}\quad
h_{2}:= \binom{n}{2}(x_{1}^{2}x_{2}^{4}x_{3}y_{1}y_{3}+x_{1}^{2}x_{2}x_{3}^{4}y_{1}y_{2}-x_{1}^{5}x_{2}x_{3}y_{2}y_{3}).
$$ 
An easy calculations shows that $\ker D \cap \kk[x,y]_{(3,2,2)} = \kk x_{1}^{3}x_{2}^{2}x_{3}^{2}$, and so $f_{1}=h_{1}+c x_{1}^{3}x_{2}^{2}x_{3}^{2}$ for some $c\in \kk$. But
then we may replace $\bbb{1,n}$ by $\bbb{1,n}-cx_{1}^{2}x_{2}^{2}x_{3}^{2}\, \bbb{1,n-1}$, which has the form $x_{1}z^{n} - nx_{j}^{2}x_{k}^{2}y_{i}z^{n-1}+ \text{ terms of lower $z$-degree}$. Thus we can assume that $f_{1}=h_{1}$, and hence $f_{2}=h_{2} + c_{2}$, where $D(c_{2})=0$. It is not difficult to see that
$$
\ker D\cap \kk[x,y]_{(5,4,4)} = \kk x_{1}^{5}x_{2}^{4}x_{3}^{4}\oplus \kk x_{1}^{2}x_{2}x_{3}^{4} u_{12} \oplus \kk x_{1}^{2}x_{2}^{4}x_{3}u_{13}
$$
Subtracting from $\bbb{1,n}$ a suitable linear combination of the invariants $x_{1}^{4}x_{2}^{4}x_{3}^{4}\,\bbb{1,n-2}$, $x_{1}x_{2}x_{3}^{4} u_{12}\,\bbb{1,n-2}$ and  $x_{1}x_{2}^{4}x_{3}u_{13}\, \bbb{1,n-2}$, we can assume that $f_{2}=h_{2}$, and the claim follows.
\end{proof}

Define $S_N:=\{u_{12}, u_{13}, u_{23}, \bbb{i,n}\mid i=1,2,3, \text{ and }0\leq n \leq N\}$ and set $A_N:=\kk[S_N]\subset A$ for all $N\geq 0$, extending our definition of the subalgebras $A_{0}$ and $A_{1}$ above.  One easily sees that $A_0$ is the ring formed by the invariants of $z$-degree 0, that is,  the invariants of the induced $\Ga$-action on the hyperplane $\VVV_{\AA^7}(z) \subset {\AA^7}$. The $A_{N}$ for $N\geq 1$ yield a family of separating morphisms $\phi_N\colon {\AA^7}\to Y_N:=\Spec(A_N)$,  
and, by \name{Kuroda}'s result mentioned above, we have $A = \bigcup_{N}A_{N}$.

The following lemma is crucial.
\begin{lem}\lab{SAGBI.lem}
For all $N\geq 0$ the subalgebra $\kk[\LT(A_{N})]\subset A$ is generated by $\LT(S_{N})$. Equivalently,  $S_{N}$ is a SAGBI-basis of $A_{N}$.
\end{lem}

\begin{proof} Put $b_{i,n}:=\LT(\bbb{i,n})$ and  $m_{ij}:=\LT(u_{ij})$:
$$
b_{i,n}=  x_{i}z^{n}, \quad m_{12}=x_{1}^{3}y_{2}, \quad m_{13}=x_{1}^{3}y_{3}, \quad m_{23}= x_{2}^{3}y_{3}.
$$
(a) We first claim that the relations between these leading term are generated by
\begin{gather*}
b_{1,n}b_{1,m}b_{1,k}m_{23}-b_{2,n}b_{2,m}b_{2,k}m_{13}=0 \text{ where } 0\leq n\leq m \leq k\leq N, \text{ and } \\
b_{i,n}b_{j,m}-b_{i,n'}b_{j,m'}=0 \text{ where }0\leq n\leq m\leq N, \ m+n=m'+n'\geq1, \ i,j\in\{1,2,3\}.
\end{gather*}
This is not difficult and we leave the details to the reader.
\par\smallskip\noindent
(b) It remains to show that, when we substitute the polynomials $\bbb{i,n}$ defined in Lemma~\reff{invar.lem} in the relations above, the leading term of the result belongs to $\kk[\LT(S_{N})]$, that is:
\begin{gather*}
\LT(\bbb{1,n}\bbb{1,m}\bbb{1,k}u_{23}-\bbb{2,n}\bbb{2,m}\bbb{2,k}u_{13})\in  \kk[\LT(S_{N})], \text{ and } \\
\LT(\bbb{i,n}\bbb{j,m}-\bbb{i,n'}\bbb{j,m'})\in \kk[\LT(S_{N})].
\end{gather*}

(b1) A simple computation shows that $\bbb{1,n}\bbb{1,m}\bbb{1,k}u_{23}-\bbb{2,n}\bbb{2,m}\bbb{2,k}u_{13}$ has $z$-degree $n+m+k$ and leading term $-x_{1}^{3}x_{3}^{3}y_{2}z^{m+n+k} = 
-b_{3,n}b_{3,m}b_{3,k}m_{12}$, which is indeed in $\kk[\LT(S_{N})]$. 

(b2) A similar computation shows that $\bbb{1,n}\bbb{2,m}-\bbb{1,n'}\bbb{2,m'}$ has $z$-degree $n+m-1$ and leading term 
$(n-n') x_{1}^{3}x_{3}^{2}y_{2}z^{n+m-1}=(n-n')b_{3,n}b_{3,m-1}m_{12}$, which also belongs to $\kk[\LT(S_{N})]$.

(b3) It remains to consider $\bbb{1,n}\bbb{1,m}-\bbb{1,n'}\bbb{1,m'}$. For $m+n\leq 1$ this expression is $0$, and for $m+n\geq 2$ it  has $z$-degree $m+n-2$ and leading term 
$(nm - n'm')x_{1}^{6}x_{2}x_{3}y_{2}y_{3}z^{n+m-2}$  which is equal to 
$(mn-m'n')m_{12}m_{13}b_{2,n-1}b_{3,m-1}$ if $n>0$ and  to $(-m'n')m_{12}m_{13}b_{2,0}b_{3,n-2}$ if $n=0$, with both belonging to $\kk[\LT(S_{N})]$.
\end{proof}

For subalgebras $B_{1}\subset B_{2}\subset A$ the {\it conductor\/} is defined as usual by $[B_{1}:B_{2}]:=\{b\in B_{2}\mid bB_{2}\subset B_{1}\}$.

\begin{lem}\lab{AN.lem}
\be
\item If $f\in A$ and $\deg_{z} f \leq N$, then $f\in A_{N}$.
\item $(x_{1},x_{2},x_{3})A_{N+1}\subseteq A_{N}$.
\item \lab{ANc.lem} $[A_{N}:A_{N+1}] \cap A_{0}= (x_{1},x_{2},x_{3})A_{0}$.
\ee
\end{lem}
\begin{proof}
(a) This statement is clear for $N=0$. If $\deg_{z}f=N >0$, then $\LT(f)$ is a monomial in $\LT(S)$ of $z$-degree $N$, and thus a monomial in $\LT(S_{N})$. Now Lemma~\reff{SAGBI.lem} implies that $\LT(f)=\LT(\tilde f)$ for some $\tilde f\in A_{N}$. Thus $\deg_{z}(f-\tilde f) < N$, and the claim follows by induction.

\par\smallskip
(b) We have $\LT(x_{i}\bbb{j,N+1})=\LT(\bbb{i,1}\bbb{j,N})$, and so $\deg_{z} (x_{i}\bbb{j ,N+1}-\bbb{i,1}\bbb{j,N})\leq N$, hence 
$(x_{i}\bbb{j ,N+1}-\bbb{i1}\bbb{j,N}) \in A_{N}$ by (a), and thus  $x_{i}\bbb{j ,N+1} \in A_{N}$.

\par\smallskip
(c) Assume that $f A_{N+1}\subset A_{N}$ for some $f\in A_{0}$. Then $f \bbb{i,N+1}\in A_{N}$ for all $i$, hence $\LT(f\bbb{i,N+1})\in \LT(A_{N})$. Thus $\LT(f\bbb{i,N+1}) = \LT(f)x_{i}z^{N+1}$ is a monomial in $\LT(S_{N})$. It follows that this monomial contains at least two factors of the form $x_{j}z^{n}=\LT(\bbb{j,n})$. This implies that $\LT(f)$, as a monomial in $\LT(S_{0})$, contains a factor $x_{j}$. Hence, $\LT(f) = x_{j}\LT(\tilde f)$ for some $\tilde f \in A_{0}$, and so $f-x_{j}\tilde f \prec f$. Now the claim follows by induction since $x_{j}A_{N+1}\subset A_{N}$, by (b).
\end{proof}

\begin{lem}\lab{f2inv.lem}
If $f\in A$ is a multi-homogeneous invariant whose multi-degree is not congruent to $(k,k,k)$ modulo 3, then $f^2\in (x_1,x_2,x_3)$. In particular,
$\bbb{j,n}^{2}\in (x_{1},x_{2},x_{3})A$ for all $j\in\{1,2,3\}$, $n\geq 0$. Moreover, the radical $\pp:=\sqrt{(x_{1},x_{2},x_{3})A}$ is generated by $\{\bbb{i,n}\}$, and $A/\pp$ is a polynomial ring in 3 variables.
\end{lem}
\begin{proof}
By induction, it suffices to show that  $\LT(f^2)=\LT(h)$ where $h\in (x_1,x_2,x_3)$. But $\LT(f)$, as a monomial in $\LT(S)$, must contain a factor of the form $x_{i}$ or $x_{i}z$ since otherwise the multi-degree is congruent to $(k,k,k)$ modulo $3$. Hence, $\LT(f^{2})$ contains a factor $x_{i}$, and so $\LT(f^2)=\LT(x_{i}p)$ for some $p\in A$.

Next we remark that $u_{i}\notin \pp$, for all $i$. In fact, if $u_{i}^{k}\in (x_{1},x_{2},x_{3})A$, then $\LT(u_{i}^{k})=\LT(u_{i})^{k}$ is a monomial in $\LT(S_{0})$ containing a factor $x_{j}$ which is impossible.
Since $\bbb{i,n}\in\pp:=\sqrt{(x_{1},x_{2},x_{3})A}$, by Lemma~\reff{f2inv.lem}, it follows that $A/\pp$ is generated by the (non-zero) images of $u_{12},u_{13},u_{23}$ which are algebraically independent, because their multi-degrees are linearly independent. Thus, $A/\pp$ is a polynomial ring in 3 variables, and $\pp$ is generated by $\{\bbb{i,n}\}$.
\end{proof}

\begin{proof}[Proof of Proposition~\ref{Roberts1.prop}(\ref{xgaalg})--(\ref{R7nice})]
For (\reff{xgaalg}) we already know that $x_{1},x_{2},x_{3}\in\fxga$, hence $\sqrt{(x_1,x_2,x_3)} \subseteq \fxga$, and, by Lemma~\reff{f2inv.lem}, we have $\bbb{i,n}\in\fxga$ for all $i,n$. Now let $f\in\fxga$. Since $A = A_{0}+(\bbb{i,n})A$ we can assume that $f\in A_{0}$. Since $A_{f}$ is finitely generated there is an $N>0$ such that $A_{f}=(A_{N})_{f}$, and so $f^{k}\bbb{i,N+1}\in A_{N}$ for some $k>0$ and all $i$. Now the claim follows from Lemma~\reff{AN.lem}.
\par\smallskip
The first part of (\reff{zeroset}) follows from Lemma~\reff{f2inv.lem}.
For the second, we look at the chain of ideals in $A$
$$
\pp_{2}:=\sqrt{(x_{1})} \subset \pp_{1}:= \sqrt{(x_{1},x_{2})} \subset \pp=\sqrt{(x_{1},x_{2},x_{3})}
$$
and the corresponding closed subschemes 
$$
Z:=\VVV_{{\AA^7}\quot\Ga}({\pp})\subset Z_{1}:=\VVV_{{\AA^7}\quot\Ga}({\pp_{1}})\subset 
Z_{2}:=\VVV_{{\AA^7}\quot\Ga}({\pp_{2}})\subset {\AA^7}\quot \Ga.
$$
It follows that
$U_{1}:=Z_{1}\cap ({\AA^7}\quot\Ga)_{x_{3}}$ is irreducible of dimension 4, because its inverse image in ${\AA^7}$ is $\VVV_{{\AA^7}}(x_{1},x_{2})\cap (\AA^7)_{x_{3}}\simeq (\AA^{5})_{x_{3}}$. Similarly, $U_{2}:=Z_{2}\cap ({\AA^7}\quot\Ga)_{x_{3}}$ is irreducible of dimension 5. Now $U_{1}$ is affine and open in $Z_{1}$ and thus the complement $\overline{U_{1}}\setminus U_{1}$ has dimension $\dim U_{1}-1$, because $Z_{1}$ is a \name{Krull} scheme. But $\overline{U_{1}}\setminus U_{1} \subseteq Z$, hence equal to $Z$, and so $\overline{U_{1}}$ contains $Z$ and is irreducible of dimension $4$. Finally, $\overline{U_{2}}$ is irreducible of dimension 5 and contains $\overline{U_{1}}$. Thus we have the chain of irreducible closed subschemes 
$$
Z \subsetneq \overline{U_{1}}\subsetneq \overline{U_{2}} \subsetneq {\AA^7}\quot \Ga.
$$
and the claim follows.
\par\smallskip
By (\reff{xgaalg}) and (\reff{zeroset}) the quotient $\AA^7\quot \Ga$ is the disjoint union of an open and a closed algebraic variety, $(\AA^{7}\quot\Ga)_{\alg}$ and $\VVV_{\AA^{7}\quot\Ga}(x_{1},x_{2},x_{3})\simeq \AA^{3}$, which clearly implies the claim.
\end{proof}

\par\bigskip
%

\end{document}